\documentclass[11pt,reqno]{amsart}

\usepackage{graphicx}
\usepackage{mathrsfs}
\usepackage{color}
\usepackage{comment}
\usepackage{amssymb}
\usepackage{hyperref}

\numberwithin{equation}{section}

\newtheorem{prop}{Proposition}[section]
\newtheorem{theo}[prop]{Theorem}
\newtheorem*{theo*}{Theorem}
\newtheorem{lemm}[prop]{Lemma}

\newtheorem*{claim}{Claim}
\newtheorem{defi}[prop]{Definition}
\newtheorem{rema}[prop]{Remark}

\theoremstyle{definition}

\newtheorem{conj}[prop]{Conjecture}

\newcommand{\RR}{\mathbf{R}}

\newcommand{\ZZ}{\mathbf{Z}}

\newcommand{\cB}{\mathcal B}

\newcommand{\cD}{\mathcal D}

\newcommand{\cF}{\mathcal F}

\newcommand{\cH}{\mathcal H}
\newcommand{\cI}{\mathcal I}

\newcommand{\cL}{\mathcal L}

\newcommand{\sE}{\mathscr{E}}

\newcommand{\uu}{\mathbf{u}}
\newcommand{\vv}{\mathbf{v}}
\newcommand{\ww}{\mathbf{w}}

\DeclareMathOperator{\proj}{proj}

\DeclareMathOperator{\spt}{spt}

\DeclareMathOperator{\dist}{dist}

\DeclareMathOperator{\Ric}{Ric}

\DeclareMathOperator{\Div}{div}

\DeclareMathOperator{\secondfund}{II}

\newcommand{\pa}[2]{\frac{\partial #1}{\partial #2}}

\newcommand{\td}[2]{\frac{d #1}{d #2}}
\newcommand{\bangle}[1]{\left\langle #1 \right\rangle}

\newcommand{\ep}{\varepsilon}

\numberwithin{equation}{section}

\begin{document}
	
	\title[A polyhedron comparison theorem in positive scalar curvature]{A polyhedron comparison theorem for 3-manifolds with positive scalar curvature}
	
	\author{Chao Li}
	\address{Department of Mathematics, Stanford University}
	\email{rchlch@stanford.edu}
	
	\begin{abstract}
		The study of comparison theorems in geometry has a rich history. In this paper, we establish a comparison theorem for polyhedra in 3-manifolds with nonnegative scalar curvature, answering affirmatively a dihedral rigidity conjecture by Gromov. For a large collections of polyhedra with interior non-negative scalar curvature and mean convex faces, we prove the dihedral angles along its edges cannot be everywhere less or equal than those of the corresponding Euclidean model, unless it is a isometric to a flat polyhedron. 
	\end{abstract}
	
	\maketitle
	
	\section{Introduction}
	A fundamental question in differential geometry is to understand metric/measure properties of Riemannian manifolds under global curvature conditions, and study notions of curvature lower bounds in spaces with low regularity. Such goals are usually achieved via geometric comparison theorems. The quest started with Alexandrov \cite{Aleksandrov51comparison}, who introduced the notion of \textit{sectional} curvature lower bounds for metric spaces via geometric comparison theorems for geodesic triangles. Similar questions for \textit{Ricci} curvature have also attracted a wide wealth of research recently (Cheeger-Colding-Naber theory; see, e.g., \cite{CheegerColding97, CheegerColding00a, CheegerColding00b, ColdingNaber12, CheegerNaber13}; for an optimal transport approach, see, e.g., \cite{LottVillani09} \cite{Sturm06a, Sturm06b, Sturm06c}).
	
	The case of scalar curvature lower bounds, however, is not as well established, possibly due to a lack of satisfactory geometric comparison theory. The first progress in this direction was made by Shi-Tam \cite{ShiTam02positive}, who proved a total boundary mean curvature comparison theorem for regions in manifolds with nonnegative scalar curvature. However, it requires a presumption of the existence of boundary isometric embedding into Euclidean spaces, which is not satisfied for general domains.
	
	As triangles play an essential role in the comparison theorems for sectional curvature, Gromov \cite{Gromov14dirac} suggested that Riemannian \textbf{polyhedra} should be of particular importance for the study of scalar curvature. In this paper, we place our focus specifically in three dimensions, and make the following
	
	\begin{defi}
		Let $P$ be a flat polyhedron in $\RR^3$. A closed Riemannian manifold $M^3$ with non-empty boundary is called a $P$-type polyhedron, if it admits a Lipschitz diffeomorphism $\phi:M\rightarrow P$, such that $\phi^{-1}$ is smooth when restricted to the interior, the faces and the edges of $P$. We thus define the faces, edges and vertices of $M$ as the image of $\phi^{-1}$ when restricted to the corresponding objects of $P$.
	\end{defi}
	
	The first case that Gromov investigated was cube-type polyhedra in three-manifolds with nonnegative scalar curvature ($P=[0,1]^3\subset \RR^3$). Let $(M^n,g)$ be a cube-type polyhedron with faces $F_j$. Let $\measuredangle_{ij}(M,g)$ denote the (possibly nonconstant) dihedral angle between two adjacent faces $F_i$ and $F_j$. Then Gromov proposed that $(M,g)$ cannot simultaneously satisfy:
	\begin{enumerate}
		\item the scalar curvature $R(g)\ge 0$;
		\item each $F_i$ is mean convex;
		\item for all pairs $(i,j)$, $\measuredangle_{ij}(M,g)<\frac{\pi}{2}$.
	\end{enumerate}
	
	Notice that conditions (2) and (3) above may be interpreted as $C^0$ properties of $g$. In fact, a face $F$ is strictly mean convex if and only if it is locally one-sided area minimizing: for any outward compactly support small perturbation $F'$ of $F$, we have $|F|<|F|'$; the dihedral angle can be measured with the metric $g$.
	
	The crucial and elegant observation from Gromov is that, if such a cube exists, then by ``doubling'' $M$ three times across the front, the right and the bottom faces, the new cube $\tilde{M}$ has isometric opposite faces. Then we identify the opposite faces of $\tilde{M}$ and obtain a torus $T^3$ with a singular metric $\tilde{g}$. Due to the geometric assumptions, the metric $\tilde{g}$ has positive scalar curvature away from a stratified singular set $S=F^2\cup L^1\cup V^0$, where:
	\begin{enumerate}
		\item $\tilde{g}$ is smooth on both sides from $F^2$. The mean curvatures of $F^2$ from two sides satisfy a positive jump;
		\item $\tilde{g}$ is an edge metric along $L$ with angle less than $2\pi$;
		\item $\tilde{g}$ is bounded measurable across isolated vertices $V^0$.
	\end{enumerate}
	
	It is known that condition (1) above implies that $\tilde{g}$ has positive scalar curvature on $F^2$ in a weak sense: a Yamabe nonpositive manifold cannot support any metric which is singular along a hypersurface satisfying the ``positive jump of mean curvature'' assumption, and has positive scalar curvature on its regular part \cite{Miao02positive}\cite{ShiTam16scalar}. The affect of condition (2) and (3) above on the Yamabe type of a manifold was investigated by C. Mantoulidis and the author in a recent paper \cite{LiMantoulidis2017}. We proved that in dimension $3$, skeleton singularities with cone angle less than $2\pi$ do not effect the Yamabe type. We refer the readers to these papers and the references therein for more details.
	
	This idea of Gromov relies on the fact that cubes are the fundamental domains of the $\ZZ^3$ actions on $\RR^3$, hence is not applicable to general polyhedra. An interesting question is then: which types of polyhedra share properties like those observed by Gromov for cube-type polyhedra in manifolds with nonnegative scalar curvature? 
	
	Another related question concerns the rigidity: what types of polyhedra are ``mean convexly extremetal''? Surprisingly, this question is unsettled even in the Euclidean spaces:
	
	\begin{conj}[Dihedral rigidity conjecture, section 2.2 of \cite{Gromov14dirac}]\label{dihedral.rigidity.conjecture}
		Let $P\in \RR^n$ be a convex polyhedron with faces $F_i$. Let $P'\subset \RR^n$ be a $P$-type polyhedron with faces $F_i'$. If
		\begin{enumerate}
			\item each $F_i'$ is mean convex, and
			\item the dihedral angles satisfy $\measuredangle_{ij}'(P')\le \measuredangle_{ij}(P)$,
		\end{enumerate}
		then $P'$ is flat.
	\end{conj}
	
	The primary scope of this paper is to answer affirmatively this conjecture for a large collection of polyhedral types in three-manifolds with nonnegative scalar curvature. We also obtain a comparison theorem for Riemannian polyhedra. 
	
	Let us define two general polyhedron types.
	\begin{defi}\label{definition.cone.prism}
		\begin{enumerate}
			\item Let $k\ge 3$ be an integer. In $\RR^3$, let $B\subset\{x^3=0\}$ be a convex $k$-polygon, and $p\in \{x_3=1\}$ be a point. Call the set
			\[\{tp+(1-t)x:t\in [0,1],x\in B\}\]
			a $(B,p)$-cone. Call $B$ the base face and all the other faces side faces.
			\item Let $k\ge 3$ be an integer. In $\RR^3$, let $B_1\subset \{x^3=0\}, B_2\subset\{x_3=1\}$ be two \underline{similar} convex $k$-polygons whose corresponding edges are parallel (i.e. the polygons are congruent up to scaling but not rotation). Call the set
			\[\{tp+(1-t)q:t\in [0,1], p\in B_1,q\in B_2\}\]
			a $(B_1,B_2)$-prism. Call $B_1,B_2$ the base faces and all the other faces side faces.
		\end{enumerate}
		If $(M,g)$ is a Riemannian polyhedron of $P$-type, where $P$ is a $(B,p)$-cone (or a $(B_1,B_2)$-prism), we call $(M,g)$ is of cone type (prism type, respectively).
	\end{defi}
	
	\begin{figure}[htbp]
		\centering
		\includegraphics{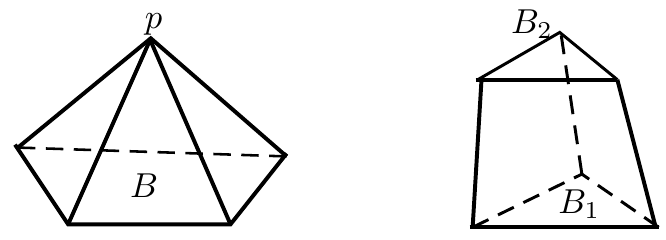}
		\caption{A $(B,p)$-cone and a $(B_1,B_2)$-prism.}
		\label{pic.cone.and.prism}
	\end{figure}
	
	The major objects we consider are Riemannian polyhedra $(M^3,g)$ of cone type or prism type, as in Definition \ref{definition.cone.prism}. Let us fix some notations that will be used throughout the paper. We use $F_1,\cdots,F_k$ to denote the side faces of $M$; if $M$ is of cone type, we use $p$ to denote the cone vertex, and $B$ to denote its base face; if $M$ is of prism type, we use $B_1,B_2$ to denote its two bases. Let $F=\cup_{j=1}^k F_j$ be the union of all side faces. Our first theorem makes a comparison between Riemannian polyhedra with nonnegative scalar curvature and their Euclidean models:
	
	\begin{theo}\label{theorem.comparison.nonrigid}
		Let $(M^3,g)$ be a Riemannian polyhedron of $P$-type with side faces $F_1,\cdots,F_k$, where $P\subset \RR^3$ is a cone or prism with side faces $F_1',\cdots,F_k'$. Denote $\gamma_j$ the angle between $F_j'$ and the base face of $P$ (if $P$ is a prism, fix one base face). Assume that everywhere along $F_j\cap F_{j+1}$,
		\begin{equation}\label{angle.assumption}
		|\pi-(\gamma_j+\gamma_{j+1})|<\measuredangle(F_j,F_{j+1}).
		\end{equation}
		Then the strict comparison statement holds for $(M,g)$. Namely, if $R(g)\ge 0$, and each $F_j$ is mean convex, then the dihedral angles of $M$ cannot be everywhere less than those of $P$.
	\end{theo}
	
	Our theorem should be contextualized in the rich history of the study of comparison theorems in differential geometry. In fact, it is not hard to argue as in \cite{Gromov14dirac} that the converse is also true: on a three-manifold with negative scalar curvature, one may construct a polyhedron which entirely invalidates the conclusions of Theorem \ref{theorem.comparison.nonrigid}. Thus the metric properties introduced by Theorem \ref{theorem.comparison.nonrigid} characterize $R(g)\ge 0$ faithfully, and may very well serve as a definition of $R(g)\ge 0$ for a metric $g$ that is only continuous.
	
	A more refined analysis enables us to characterize the rigidity behavior for Theorem \ref{theorem.comparison.nonrigid}, thus answering Conjecture \ref{dihedral.rigidity.conjecture} for cone type and prism type polyhedra, with the very mild a priori angle assumptions \eqref{angle.assumption}. In fact, we obtain:
	
	\begin{theo}\label{theorem.rigidity}
		Under the same assumptions of Theorem \ref{theorem.comparison.nonrigid} and the extra assumption that
		\begin{equation}\label{extra.angle.assumption}
		\gamma_j\le \pi/2, j=1,2,\cdots,k,\quad \text{or}\quad \gamma_j\ge \pi/2, j=1,2,\cdots,k,
		\end{equation}
		we have the rigidity statement. Namely, if $R(g)\ge 0$, each $F_j$ is mean convex, and $\measuredangle_{ij}(M,g)\le \measuredangle_{ij}(P,g_{Euclid})$, then $(M,g)$ is isometric to a flat polyhedron in $\RR^3$.
	\end{theo}
	
	The angle assumption \eqref{angle.assumption} may be regarded as a mild regularity assumption on $(M,g)$. It is satisfied, for instance, by any small $C^0$ perturbation of the Euclidean polyhedron $P$. Moreover, assumption \eqref{angle.assumption} is vacuous, if all the angles $\gamma_j$ are $\pi/2$. In this case, the Euclidean model is a prism with orthogonal base and side faces, and we are able to obtain the prism inequality in section 5.4, \cite{Gromov14dirac} as a special case.
	
	Motivated by the Schoen-Yau dimension reduction argument \cite{SchoenYau79structure}, we have also been able to generalize Theorem \ref{theorem.comparison.nonrigid} and Theorem \ref{theorem.rigidity} in higher dimensions. They will appear in a forthcoming paper.
	
	Now we indicate the strategy of the proof for Theorem \ref{theorem.comparison.nonrigid} and Theorem \ref{theorem.rigidity} and the organization of the paper. Consider the following energy functional:
	\begin{equation}\label{energy.functional}
	\cF(E)=\cH^2(\partial E\cap \mathring{M})-\sum_{j=1}^k(\cos\gamma_j)\cH^2(\partial E\cap F_j),
	\end{equation}
	and the variational problem
	\begin{equation}\label{capillary.variational.problem}
	\cI=\inf\{\cF(E): E\in \sE\},
	\end{equation}
	here $\sE$ is the collection of contractible open subset $E'$ such that: if $M$ is of cone type, then $p\in E'$ and $E'\cap B=\emptyset$; if $M$ is of prism type, then $B_2\subset E'$ and $E'\cap B_1=\emptyset$. If the solution to \eqref{capillary.variational.problem} is regular, its boundary $\Sigma^2=\partial E\cap \mathring{M}$ is called a capillary minimal surface. That is, $\Sigma$ is a minimal surface that contacts each side face $F_j$ at constant angle $\gamma_j$. The existence, regularity and geometric properties of capillary surfaces have attracted a wealth of research throughout the rich history of geometric variational problems. We refer the readers to the book of Finn \cite{Finn1986equilibrium} for a beautiful and thorough introduction.
	
	Our first observation is that $\cI$ is always finite: since $M$ is compact, we deduce that
	\[\cI\ge -\sum_{j=1}^k(\cos\gamma_j)\cH^2(F_j)>-\infty.\]
	Thus a minimizing sequence exists. The existence and boundary regularity of the solution to \eqref{capillary.variational.problem} was treated by Taylor \cite{Taylor77boundaryregularity} (see page 328-(6); see also the discussion for more general anisotropic capillary problems by De Philippis-Maggi \cite{DePhilippisMaggi15regularity}). Using the language of integral currents, Taylor proved the existence of the minimizer $\Sigma$, and that $\Sigma$ is $C^\infty$ regular up to its boundary, where $\partial M$ is smooth. However, the variational problem \eqref{capillary.variational.problem} has obstacles: the base face(s) of $M$. To overcome this difficulty, we apply the interior varifold maximum principle \cite{SolomonWhite89maximum} and a new boundary maximum principle, and reduce \eqref{capillary.variational.problem} to a variational problem without obstacles. We then adapt ideas from Simon \cite{Simon80regularity} and Lieberman \cite{Lieberman88Holder}, and obtain a $C^{1,\alpha}$ regularity property of $\Sigma$ at its corners. This is the only place we need to use the angle assumption \eqref{angle.assumption}. The existence and regularity of $\Sigma$ is established in section 2. In section 3, we unveil the connection between interior scalar curvature, the boundary mean curvature and the dihedral angle captured by the variational problem \eqref{capillary.variational.problem}, and derive various geometric consequences with $\Sigma$. We prove Theorem \ref{theorem.comparison.nonrigid} with the second variational inequality and the Gauss-Bonnet formula. We then proceed to section 4 for the proof of Theorem \ref{theorem.rigidity}, where an analysis for the ``infinitesimally rigid'' minimal capillary surface $\Sigma$ is carried out, with the idea pioneered by Bray-Brendle-Neves \cite{BrayBrendleNeves10rigidity}. The new challenge here is to deal with the case when $\cI=0$. We develop a new general existence result of constant mean curvature capillary foliations near the vertex $p$, and establish the dynamical behavior of such foliations in nonnegative scalar curvature.

	\noindent \textbf{Acknowledgement:} The author wishes to thank Rick Schoen, Brian White, Leon Simon, Rafe Mazzeo, Or Hershkovits and Christos Mantoulidis for stimulating conversations. He also wishes to thanks the referee for greatly improving the exposition. Part of this work was carried out when the author was visiting the University of California, Irvine. He wants to thank Department of Mathematics, UCI, for their hospitality.

	\section{Existence and regularity}
	We discuss the existence and regularity of the minimizer for the variational problem \eqref{capillary.variational.problem}. The goal of this section is:
	\begin{theo}\label{theorem.existence.regularity}
		Consider the variational problem \eqref{capillary.variational.problem} in a Riemannian polyhedron $(M^3,g)$ of cone or prism type. Assume $\cI<0$ if $M$ is of cone type. Then $\cI$ is achieved by an open subset $E$. Moreover, $\Sigma=E\cap \mathring{M}$ is an area minimizing surface, $C^{1,\alpha}$ to its corners for some $\alpha>0$, and meets $F_j$ at constant angle $\gamma_j$.
	\end{theo}
	
	We first introduce some notations and basic geometric facts on capillary surfaces. Then we reduce the obstacle problem \eqref{capillary.variational.problem} equivalently to a variational problem without any obstacle. This is done via a varifold maximum principle. Hence the regularity theory developed in \cite{Taylor77boundaryregularity} is applicable, and we get regularity in $\mathring{\Sigma}$,  and in $\partial \Sigma$ in $\mathring{F_j}$. The regularity at the corners of $\Sigma$ is then studied with an idea of Simon \cite{Simon80regularity}. At the corner, we prove that the surface is graphical over its planar tangent cone. Then we invoke the result of Lieberman \cite{Lieberman88Holder}, which showed that the unit normal vector field is H\"older continuous up to the corners.
	
	\subsection{Preliminaries}
	We start by discussing some geometric properties of capillary surfaces. In particular, we deduce the first and second variation formulas for the energy functional \eqref{energy.functional}. Let us fix some notation.
	
	Let $P$ be an orientable Riemannian manifold of dimension $p$ and $M$ a closed compact polyhedron of cone or prism types in $N$. Let $\Sigma^{n-1}$ be an orientable $n-1$ dimensional compact manifold with non-empty boundary $\partial \Sigma$ and $\partial \Sigma\subset \partial M$. We denote the topological interior of a set $U$ by $\mathring{U}$. Assume $\Sigma$ separates $\mathring{M}$ into two connected components. Fix one component and call it $E$. Denote $X$ the outward pointing unit normal vector field of $\partial M$ in $M$, $N$ the unit normal vector field of $\Sigma$ in $E$ pointing into $E$, $\nu$ the outward pointing unit normal vector field of $\partial \Sigma$ in $\Sigma$, $\overline{\nu}$ the unit normal vector field of $\partial\Sigma$ in $\partial M$ pointing outward $E$. Let $A$ denote the second fundamental form of $\Sigma\subset E$, $\secondfund$ denote the second fundamental form of $\partial M\subset M$. We take the convention that $A(X_1,X_2)=\bangle{\nabla_{X_1}X_2,N}$. Denote $H,\overline{H}$ the mean curvature of $\Sigma\subset E$, $\partial M\subset M$, respectively. Note that in our convention, the unit sphere in $\RR^3$ has mean curvature $2$.
	
	\begin{figure}[htbp]
		\centering
		\includegraphics{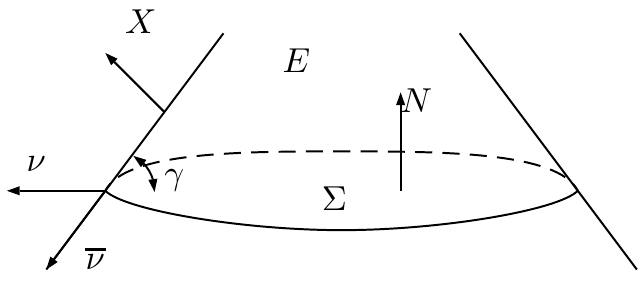}
		\caption{Capillary surfaces}
		\label{picture.capillary.notations}
	\end{figure}
	
	By an admissible deformation we mean a diffeomorphism $\Psi:(-\ep,\ep)\times \Sigma\rightarrow M$ such that $\Psi_t:\Sigma\rightarrow M$, $t\in (-\ep,\ep)$, defined by $\Psi_t(q)=\Psi(t,q)$, $q\in \Sigma$, is an embedding satisfying $\Psi_t(\Sigma)\subset \mathring{M}$ and $\Psi_t(\partial \Sigma)\subset \partial M$, and $\Psi_0(x)=x$ for all $x\in \Sigma$. Denote $\Sigma_t=\Psi_t(\Sigma)$. Let $E_t$ be the corresponding component separated by $\Sigma_t$. Denote $Y=\pa{\Psi (t,\cdot)}{t}\vert_{t=0}$ the vector field generating $\Psi$. Then $Y$ is tangential to $\partial M$ along $\partial\Sigma$. Fix the angles $\gamma_1,\cdots,\gamma_k\in (0,\pi)$ on the faces $F_1,\cdots,F_k$ of $M$. Consider the energy functional
	\[F(t)=\cH^{n-1}(\Sigma_t)-\sum_{j=1}^k(\cos \gamma_j)\cH^{n-1}(\partial E_t\cap F_j).\]
	
	We now deduce the first variation formula of $F(t)$. Let $f=\bangle{Y,N}$ be the normal component of the vector field $Y$. By the usual first variation formula on volume function and integration by parts,
	\[\td{}{t}\bigg\vert_{t=0}\cH^{n-1}(\Sigma_t)=\int_\Sigma \Div_\Sigma Y d\cH^{n-1}=-\int_\Sigma H fd\cH^{n-1}+\int_{\partial \Sigma} \bangle{Y,\nu} d\cH^{n-2}.\]
	
	On the other hand, for each $j$, $1\le j\le k$, 
	\[\td{}{t}\bigg\vert_{t=0} -(\cos\gamma_j) \cH^{n-1} (\partial E_t\cap F_j)=-\cos\gamma_j\int_{\partial \Sigma\cap F_j} \bangle{Y,\overline{\nu}} d\cH^{n-2}.\]
	
	Adding the above two equations, the first variation of $F(t)$ is given by
	\begin{equation}\label{first.variation.for.regular.surfaces}
	\td{}{t}\bigg\vert_{t=0}F(t)=-\int_\Sigma Hf d\cH^{n-1}+\sum_{j=1}^k\int_{\partial \Sigma\cap F_j}\bangle{Y,\nu-(\cos\gamma_j) \overline{\nu}}d\cH^{n-2},
	\end{equation}
	We note that \eqref{first.variation.for.regular.surfaces} holds more generally in the context of varifolds, see \eqref{capillary.varifold.first.variation}. Also, in the first variation \eqref{first.variation.for.regular.surfaces}, there is no contribution from the corners of $\Sigma$. The surface $\Sigma$ is said to be minimal capillary if $F'(t)=0$ for any admissible deformations. If follows from \eqref{first.variation.for.regular.surfaces} that $\Sigma$ is minimal capillary if and only if $H\equiv 0$ and $\nu-(\cos\gamma_j)\overline{\nu}$ is normal to $F_j$; that is, along $F_j$ the angle between the normal vectors $N$ and $X$, or equivalently, between $\nu$ and $\overline{\nu}$, is everywhere equal to $\gamma_j$.
	
	Assume $\Sigma$ is minimal capillary. We then have the second variational formula:
	\begin{multline}\label{second.variation.formula}
	\td{^2}{t^2}\bigg\vert_{t=0}F(0)=-\int_{\Sigma}(f\Delta f+(|A|^2+\Ric(N,N))f^2)d\cH^{n-1}\\
	+\sum_{j=1}^k\int_{\partial\Sigma\cap F_j}f\left(\pa{f}{\nu}-Qf \right)d\cH^{n-2},
	\end{multline}
	where on $\partial\Sigma \cap F_j$,
	\[Q=\frac{1}{\sin \gamma_j}\secondfund(\overline{\nu},\overline{\nu})+\cot\gamma_j A(\nu,\nu).\]
	Here $\Delta$ is the Laplace operator of the induced metric on $\Sigma$, and $\Ric$ is the Ricci curvature of $M$. For a proof of the second variation formula, we refer the readers to the appendix of \cite{RosSouam97capillarystability}.

	\subsection{Maximum principles}
	We first observe that \eqref{capillary.variational.problem} is a variational problem with \textit{obstacles}: $E\cap B_1=\emptyset$ if $M$ is of cone type, and $B_2\subset E$, $E\cap B_1=\emptyset$ if $M$ is of prism type. To apply the existence and the regularity theories of Taylor \cite{Taylor77boundaryregularity}, we first prove that it suffices to consider a variational problem \textit{without} any obstacles. Such reduction is usually achieved via varifold maximum principles, see e.g. \cite{SolomonWhite89maximum,white10maximumcodimension}\cite{LiZhou17freeboundarymaximum}. In our case, the maximum principles hinge upon the special structure of the obstacle: that $B$ (or $B_1,B_2$) is mean convex, and that the dihedral angles along $\partial F_j\cap B$ are nowhere larger than $\gamma_j$. In fact, if $\Sigma=\partial E\cap\mathring{M}$ is a $C^1$ surface with piecewise smooth boundary, then it is not hard to see from the first variational formula \eqref{first.variation.for.regular.surfaces} that
	\begin{itemize}
		\item $\Sigma$ and $B$ do not touch in the interior.
		\item $\partial \Sigma$ does not contain any point on $ F_j\cap B$ where the dihedral angle is strictly less than $\gamma_j$.
	\end{itemize}
	Thus $\Sigma$ is a minimal surface that meets each $F_j$ at constant angle $\gamma_j$.
	
	The interior maximum principle has been investigated in different scenarios \cite{Simon87strict}\cite{SolomonWhite89maximum}\cite{white10maximumcodimension}\cite{Wickramasekera14strongmaximum}. Notice that the energy minimizer of \eqref{capillary.variational.problem} is necessarily area minimizing in the interior. We apply the strong maximum principle by Solomon-White \cite{SolomonWhite89maximum} and conclude that the surface $\Sigma=\partial E\cap \mathring{M}$ cannot touch the base face $B$, unless lies entirely in $B$. 
	
	Here we develop a new boundary maximum principle. For the purpose of this paper, we only consider energy minimizing currents of codimension $1$ associated to \eqref{capillary.variational.problem}. However, we conjecture that a similar statement should hold for varifolds with boundary in general codimension. (See, for instance, the boundary maximum principle of Li-Zhou \cite{LiZhou17freeboundarymaximum}.)
	
	\begin{prop}\label{boundary.maximum.principle}
		Let $M$ be a polyhedron of cone type. Let $T\in\cD^2(M), E\in \cD^3(U)$ be rectifiable currents with $T=\partial E\cap \mathring{M}$ and $\spt(\partial T)\subset F$. Assume $E$ is an energy minimizer of \eqref{capillary.variational.problem}. Then $\spt(T)$ does not contain any point on the edge $F_j\cap B$ where the dihedral angle is less than $\gamma_j$.
	\end{prop}
	
	By a similar argument, in the case that $M$ is of prism type, $\spt(T)$ does not contain any point on $F_j\cap B_1$ where the dihedral angle is less than $\pi-\gamma_j$, or any point on $F_j\cap B_2$ where the dihedral angle is less than $\gamma_j$. Combine this with the interior maximum principle, we conclude that the minimizer to \eqref{capillary.variational.problem} lies in the interior of $M$, and hence an energy-minimizer for the $\cF$ \textit{without} any barriers. Thus the existence and regularity theory developed in \cite{Taylor77boundaryregularity}\cite{DePhilippisMaggi15regularity} concludes that the minimizer $T=\partial E\cap \mathring{M}$ exists, and is regular away from the corners.

	\begin{proof}
		Assume, for the sake of contradiction, that a point $q\in F_j\cap B$ is also in $\spt{T}$, and that the dihedral angle between $F_j$ and $B$ at $q$ is less than $\gamma_j$. In the rest of proof we use $\|T\|$ to denote the associated varifold. Fix a vector field $Y$ tangential to $\partial M$, such that $Y$ is transversal along $B$ and points into $M$ at $B$. Since $T=\partial E\cap \mathring{M}$, it is a rectifiable current with multiplicity one, the first variational formula for the energy functional $\cF$ applies:
		\begin{equation}\label{variation.varifold}
		\td{}{t}\bigg\vert_{t=0}\cF(\Psi_t(E))=-\int Hf d\|T\|+\sum_j\int \bangle{Y,\nu-(\cos\gamma_j) \overline{\nu}}d\|\partial T\|,
		\end{equation}
		where $f,\overline{\nu}$ are the geometric quantities defined as before, $H$ is the generalized mean curvature of $T$, and $\nu$ is the generalized outward unit normal of $\|\partial T\|$. Since the dihedral angle between $F_j$ and $B$ at $q$ is strictly less than $\gamma_j$, we have
		\begin{equation}\label{dihedral.inquality}
		\bangle{Y,\nu'-\cos\gamma_j\overline{\nu}}>0,
		\end{equation}
		for any $\nu'$ at $q$ which is the outward unit normal vector of some two-plane in $T_q M$. By the interior maximum principle, $H\equiv 0$.\footnote{The same argument here applies to the general case where the barrier $B$ has bounded mean curvature, see Remark \ref{remark.not.necessary.mean.convex}.} Hence
		\begin{equation}\label{boundary.almost.regular}
		\|\partial T\|(\spt(T)\cap \cB)=0,\end{equation}
		where 
		\[\cB=\bigcup_j\{q\in F_j\cap B:\text{the dihedral angle at $q$ is less than $\gamma_j$.}\}\]
		
		The boundary regularity theorem of Taylor \cite{Taylor77boundaryregularity} implies that for any point $q'\in \partial T\setminus  \cB$, the current $T$ is smooth up to $q'$. In particular, the density of $T$ at $q'$ is given by $\Theta^2(\|T\|,q')=\frac{1}{2}$. Denote $W$ the two dimensional varifold $v(\partial E\cap F_j)$ associated with $E\cap F_j$, $Z=\|T\|-\cos\gamma_j W$. Since the faces $F_j$ and $B$ intersects smoothly at $q$, we have the following monotonicity formula (we delay the derivation of a more general monotonicity formula in the next section, see \eqref{monotonicity.for.corners}):
		\begin{equation}\label{boundary.monotonicity}
		\exp(cr^\alpha)\frac{\|Z\|(B_r(q))}{r^2} \text{ is increasing in $r$,}\end{equation}
		for $r$ sufficiently small, where $c$ and $\alpha>0$ depends only on the geometry of $F_j$ and $B$. It is then straightforward to check as in Theorem 3.5-(1) in \cite{Allard75boundarybehavior} that the $\theta^2(\|T\|,\cdot)$ is an uppersemicontinuous function on $\spt(T)\cap\partial T$. By \eqref{boundary.almost.regular} we then conclude
		\begin{equation}\label{boundary.lower.density.bound}
		\Theta^2(\|T\|,\cdot)\ge \frac{1}{2}>0\end{equation}
		everywhere on $\spt(T)\cap \partial T$.
		
		Consider a tangent cone $T_\infty$ of $T$ at $q$. Let $E_\infty$ be the associated three dimensional current with $T_\infty=\partial E_\infty$. By the monotonicity \eqref{boundary.monotonicity} and the lower density bound \eqref{boundary.lower.density.bound}, $T_\infty$ is a nonempty cone in $C$ through $q_\infty$, where $C$ is the region in $\RR^3$ enclosed by the two planes $F_{j,\infty}$ and $B_\infty$ intersecting at an angle $\gamma'<\gamma_j$, and where $q_\infty\in F_\infty\cap B_\infty$. By scaling, for any open set $U\subset\subset \RR^3$, $E_\infty$ minimizes the energy functional
		\begin{equation}\label{variational.problem.euclid}
		\cF_\infty(E')=\cH^2(\partial E'\cap \mathring{C}\cap U)-(\cos\gamma_j) \cH^2(\partial E'\cap F_\infty\cap U)\end{equation}
		among open sets $E'$ with $\partial E'\subset \overline{F_\infty}$. Since two-planes are the only minimal cones in $\RR^3$, $T_\infty$ is a two-plane through $q_\infty$. However, since $\measuredangle(F_\infty,\mathring{B_\infty})<\gamma_j$, no two-plane through $q_\infty$ can be the minimizer of \eqref{variational.problem.euclid}. Contradiction.
		
	\end{proof}
	
	\begin{rema}\label{remark.not.necessary.mean.convex}
		The above proof only uses the fact that $T$ is minimal in a very weak manner. In fact, the same argument holds under the assumption that the generalized mean curvature $H$ is bounded measurable. This is implied, for instance, by that the barrier $\overline{B}$ has bounded mean curvature (instead of being mean convex).
	\end{rema}
	
	\begin{rema}
		The fact that $T$ is energy minimizing is only used to guarantee the existence of an area minimizing tangent cone. Motivated by \cite{SolomonWhite89maximum}, we speculate that a similar statement should hold for varifolds with boundary that are stationary for the energy functional \eqref{capillary.variational.problem}.
	\end{rema}

	\subsection{Regularity at the corners}
	We proceed to study the regularity of the minimizer $T=\partial E\cap \mathring{M}$ at the corners. Since $T$ is regular away from the corners, our idea is to adapt the argument of Simon \cite{Simon80regularity}, and prove $\spt(T)$ is graphical near a corner. We refer the readers to \cite{Simon80regularity} for full details. Then we apply the theorem of Lieberman \cite{Lieberman88Holder} to conclude that $\spt(T)$ has a H\"older continuous unit normal vector field to its corners. 
	
	Consider any two adjacent side faces $F_j,F_{j+1}$ and let $L=F_j\cap F_{j+1}$. Without loss of generality let $j=1$. Fix a point $q\in \spt(T)\cap L$. Let $\theta$ be the angle between $F_1$ and $F_2$ at $q$. Recall that we assume 
	\begin{equation*}
	|\pi-(\gamma_1+\gamma_2)|<\theta\le\theta',
	\end{equation*}
	where $\theta'$ is the (constant) dihedral angle between corresponding faces in the Euclidean model. To start the regularity discussion, we first make the following simple calculation in Euclidean space.
	
	\begin{lemm}\label{lemma.Euclidean.angle.condition}
		Let $\Gamma_1,\Gamma_2$ be two half-planes in $\RR^3$, enclosing a wedge region $W$ with opening angle $\theta'\in (0,\pi)$. Suppose $\Gamma$ is a plane in $\RR^3$, such that the dihedral angle between $\Gamma$ and $\Gamma_i$, $i=1,2$, is $\gamma_i\in (0,\pi)$. Then we have
		\[|\pi-(\gamma_1+\gamma_2)|<\theta'<\pi-|\gamma_1-\gamma_2|.\]
	\end{lemm}
	\begin{proof}
		Let $\nu_i$ be the unit normal vector of $\Gamma_i$ pointing out of $W$, $\nu$ is the unit normal vector of $\Gamma$. Then by assumption, we have that
		\[\nu_1\cdot \nu_2=-\cos \theta',\quad \nu_i\cdot \nu=\cos \gamma_i, \quad i=1,2.\]
		We calculate
		\[(\nu_1\times \nu)\cdot (\nu_2\times \nu)=\nu_1\cdot \nu_2-(\nu_1\cdot\nu)(\nu_2\cdot \nu)=-\cos\theta'-\cos\gamma_1\cos\gamma_2.\]
		On the other hand, we have that $|\nu_1\times \nu|=\sin\gamma_i$, $i=1,2$. Hence $|(\nu_1\times \nu)\cdot (\nu_2\times \nu)|\le \sin\gamma_1\sin\gamma_2$. Since $\Gamma$ intersects $\Gamma_1,\Gamma_2$ transversely, we actually have $|(\nu_1\times \nu)\cdot (\nu_2\times \nu)|< \sin\gamma_1\sin\gamma_2$. By a simple calculation, this implies that
		\[\cos(\gamma_1+\gamma_2)<\cos(\pi-\theta')<\cos(\gamma_1-\gamma_2).\]
		Thus $|\pi-(\gamma_1+\gamma_2)|<\theta'<\pi-|\gamma_1-\gamma_2|$, as desired.
	\end{proof}

	We therefore may assume that 
	\begin{equation}\label{recall.angle.condition}
	|\pi-(\gamma_1+\gamma_2)|<\theta<\pi-|\gamma_1-\gamma_2|.
	\end{equation}
	As an immediate observation following Lemma \ref{lemma.Euclidean.angle.condition}, \eqref{recall.angle.condition} is a necessary condition for the regularity statement in Theorem \ref{theorem.existence.regularity}\footnote{When condition \eqref{recall.angle.condition} is not satisfied, we conjecture that there will be a ``cusp" singularity forming at the corner. For instance, see (0.4) and (0.5) in \cite{Simon80regularity}, and the discussion therein.}. Precisely, if the capillary surface $\Sigma$ is $C^{1,\alpha}$ regular up to the corners, its tangent plane at the corner satisfies the assumption of Lemma \ref{lemma.Euclidean.angle.condition}, imposing the range for $\theta$ as in \eqref{recall.angle.condition}. To prove Theorem \ref{theorem.existence.regularity}, we verify that condition \eqref{recall.angle.condition} is also sufficient to guarantee the regularity of $\Sigma$.
	
	For $\rho>0$, denote $C_\rho=\{x\in M:\dist_M(x,L)<\rho\}$, $B_\rho=\{x\in M:\dist_M(x,q)<\rho\}$. In this section, and subsequently, let $c$ be a constant that may change from line to line, but only depend on the geometry of the polyhedron $M$. Our argument is parallel to that of \cite{Simon80regularity}: we prove a uniform lower density bound around $q$, and consequently analyze the tangent cone at $q$. 
	
	\subsubsection{Lower density bound}
	Our first task is to establish an upper bound for the area of $T$. Precisely, we prove:
	\begin{lemm}\label{lemma.rough.H^2.estimate}
		For $\rho$ small enough, $\cH^2(T\cap C_\rho)\le c\rho$.
	\end{lemm}
	
	\begin{proof}
		This is straightforward consequence of the fact that $T=\partial E\cap \mathring{M}$ minimizes the energy $\cF$. In fact, for any open subset $U\subset \subset M$, $E$ minimizes the functional 
		\[\cH^2(E'\cap \mathring{M}\cap U)-\sum_{j=1}^k(\cos\gamma_j)\cH^2(\partial E'\cap \partial M\cap U)\]
		among all sets $E'\subset M$ with finite perimeter, $p$ (or $B_2$)$\subset E'$, $E'\cap B=\emptyset$. In particular, choose $E'$ to be a small open neighborhood of $p$ when $M$ is a $(B,p)$-cone, and a small tubular neighborhood of $B_2$ when $M$ is a $(B_1,B_2)$-prism. Let $T'=\partial E'\cap \mathring{M}$. Choose $U=C_\rho$. We conclude that 
		\begin{multline}
		\cH^2(T\cap C_\rho)-\sum_{j=1}^2(\cos \gamma_j)\cH^2(\partial E\cap C_\rho\cap F_j)\\
		\le \cH^2(T'\cap C_\rho)-\sum_{j=1}^2(\cos\gamma_j)\cH^2(\partial E'\cap C_\rho\cap F_j).
		\end{multline}
		By the rough estimate that
		\[\cH^2(C_\rho\cap F_j)\le c\rho \quad\text{and}\quad \cH^2(C_\rho \cap B)\le c\rho^2,\]
		we conclude the proof.
	\end{proof}
	
	Denote $\Sigma=\spt(T)\setminus L$. Since the mean curvature of $T$ is zero in its interior, from the first variation formula for varifolds, we have that, for any $C^1$ vector field $\phi$ compactly supported in $M\setminus L$,
	\begin{equation}\label{general.first.variation}
	\int_\Sigma \Div_\Sigma \phi d\cH^2=\int_{\partial\Sigma}\phi\cdot \nu d\cH^1.
	\end{equation}
	We first bound the length of $\partial \Sigma$. Precisely, let $r$ be the radial distance function $r=\dist(\cdot,L)$, let $\phi$ be any vector field, supported in $M$ with $\sup r|D\phi|<\infty$ and $C^1$ in $\mathring{M}$. (Note that we allow $\phi$ to have support on $L$.) By a standard approximation argument as in \cite{Simon80regularity}, we deduce that
	\begin{multline}\label{first.variation.phi}
	\rho^{-1}\int_{\Sigma\cap C_\rho}\phi\cdot \nabla_\Sigma r d\cH^2-\int_{\partial\Sigma}\min\left\{\frac{r}{\rho},1\right\}\phi\cdot \nu d\cH^1\\
	=-\int_\Sigma \min\left\{\frac{r}{\rho},1\right\}\Div_\Sigma \phi d\cH^2.
	\end{multline}
	
	We are going to use \eqref{first.variation.phi} in two different ways. By the angle assumption \eqref{recall.angle.condition}, $|\pi-(\gamma_1+\gamma_2)|<\theta$. Therefore in the $2$-plane $(T_qL)^\perp\subset T_qM$, there is a unit vector $\tau$ such that
	\begin{equation}\label{angle.inequality.tau1}
	(-X)\vert_{F_j}\cdot \tau>\cos\gamma_j, \quad j=1,2,
	\end{equation}
	where $(-X)\vert_{F_j}$ is the inward pointing unit normal vector of $\partial M\subset M$, restricted to the face $F_j$. Extend $\tau$ in a neighborhood of $q\in M$ as a constant vector field, and with slight abuse of notation, denote this constant vector field also by $\tau$. In \eqref{first.variation.phi}, replace $\phi$ to be the constant vector $\tau$. \eqref{angle.inequality.tau1} then implies that
	\begin{equation}\label{angle.inequality.tau2}
	(-\nu)\cdot \tau\ge c>0,
	\end{equation}
	in $\Sigma\cap C_{\rho_0}$, for sufficiently small $\rho_0$. Taking $\rho\rightarrow 0$ in \eqref{first.variation.phi} and using Lemma \ref{lemma.rough.H^2.estimate}, we deduce that
	\begin{equation}\label{boundary.length.estimate}
	\cH^1(\partial \Sigma\cap C_\rho)<\infty.
	\end{equation}
	
	The angle assumption $\theta<\pi$ then guarantees that the vector $\tau\in T_qM$ defined above also satisfies
	\begin{equation}\label{angle.inequality.tau3}
	\tau\cdot \nabla_M r\ge c>0,
	\end{equation}
	where $r$ is the radial distance function. See Figure \ref{angle.depiction.1} for an illustration of the choice of $\tau$.
	\begin{figure}[htbp]
		\centering
		\includegraphics{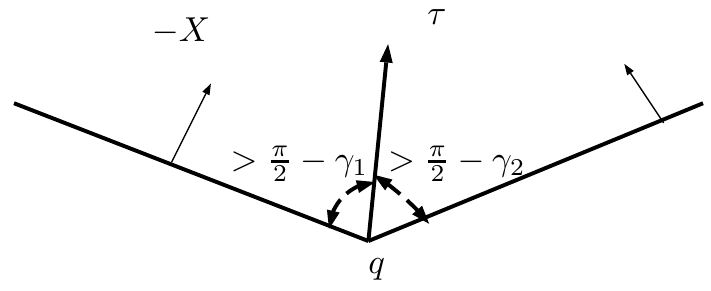}
		\caption{The choice of the vector $\tau$.}
		\label{angle.depiction.1}
	\end{figure}
	
	Now we use \eqref{first.variation.phi} in a second way. We replace $\phi$ by $\psi\tau$, where $\tau$ is the constant vector field defined in a neighborhood of $q\in M$ as above. Then by the same argument as (1.8)-(1.10) in \cite{Simon80regularity},
	\begin{multline}
	\rho^{-1}\int_{\Sigma\cap C_\rho}\psi\tau\cdot \nabla_M r d\cH^2-\int_{\partial \Sigma}\psi\tau\cdot \nu d\cH^1\\
	\le c\int_{\Sigma}(\psi+|\nabla_M \psi|)d\cH^2+o(1).
	\end{multline}
	As a consequence of \eqref{angle.inequality.tau2} and \eqref{angle.inequality.tau3},
	\begin{equation}\label{varifold.inequality.required.in.allard}
	\|\delta T\|(\psi)\le c\int (\psi+|\nabla_M \psi|)d\|T\|,
	\end{equation}
	where here we view $T$ as the associated varifold. Then we apply the isoperimetric inequality (7.1 in \cite{Allard72firstvariationofvarifold}) and the Moser type iteration (7.5(6) in \cite{Allard72firstvariationofvarifold}) as in \cite{Simon80regularity}, and conclude that
	\begin{equation}\label{lower.density.bound}
	\cH^2(T\cap B_\rho (q))\ge c\rho^2.
	\end{equation}
	\begin{rema}
		The argument above does not use the fact that $\Sigma$ is a two dimensional surface in an essential way. The same argument should work for capillary surfaces in general dimensions.
	\end{rema}
	
	\begin{rema}
		Notice that we only require the weaker angle assumption $|\pi-(\gamma_1+\gamma_2)|<\theta<\pi$ for the lower density bound. We are going to see that the assumption $\theta<\pi-|\gamma_1-\gamma_2|$ is used to classify the tangent cone.
	\end{rema}
	
	\subsubsection{Monotonicity and tangent plane}
	We proceed to derive the monotonicity formula and study the tangent cone at a point $q\in \spt(T)\cap L$. For $j=1,2$, denote $W_j=E_j\cap (F_j\setminus L)$. We also use $W_j$ to denote the associated two dimensional varifold. The divergence theorem implies that
	\begin{equation}\label{boundary.first.variation.phi}
	\delta W_j(\psi\phi)=\int_{\partial \Sigma}\psi\phi\cdot \overline{\nu},
	\end{equation}
	where recall $\overline{\nu}$ is the unit normal vector of $\partial \Sigma$ which is tangent to $\partial M$ and points outward $E$.
	
	Since $\phi$ is tangential on $F_j\setminus L$, $\nu-(\cos\gamma_j)\overline{\nu}$ is normal to $\phi$. We then multiply $-\cos\gamma_j$ to \eqref{boundary.first.variation.phi} and add the result to \eqref{general.first.variation}, thus obtaining
	\begin{equation}\label{capillary.varifold.first.variation}
	\left[\delta \|T\|-\sum_{j=1}^2(\cos \gamma_j)\delta W_j\right](\psi\phi)=0.
	\end{equation}
	Denote $Z=\|T\|-\sum_{j=1}^2 (\cos \gamma_j)W_j$. Now since $F_1,F_2$ are smooth surfaces intersecting transversely on $L$, we may choose a vector field $\phi$ that is $C^{1,\alpha}$ close to the radial vector field $\nabla_M\dist(\cdot,q)$. For a complete argument, see (2.4) of \cite{Simon80regularity}. Then by a minor modification of the argument of 5.1 in \cite{Allard72firstvariationofvarifold}, we conclude that
	\begin{equation}\label{monotonicity.for.corners}
	\exp(c\rho^\alpha)\frac{\|Z\|(B_\rho(0))}{\rho^2} \text{ is increasing in $\rho$, for $\rho<\rho_0$.}
	\end{equation}
	
	We thus deduce from \eqref{lower.density.bound} and \eqref{monotonicity.for.corners} that there is a nontrivial tangent cone $Z_\infty$ of $Z$ at $q$. Precisely, under the homothetic transformations $\mu_r$ defined by $x\mapsto r (x-q)$ ($r>0$), $(\mu_{r_k \#}T,\mu_{r_k \#}W_1,\mu_{r_k \#}W_2,\mu_{r_k \#}Z)$ subsequentially converges to $Z_\infty=\|T_\infty\|-\sum_{j=1}^2W_{j,\infty}$ in $\RR^3$. Let $F_{j,\infty}$, $j=1,2$, denote the corresponding limit planes in $\RR^3$ of $F_j$, $L_\infty=F_{1,\infty}\cap F_{2,\infty}$. Denote $P_\infty=L^\perp$ the two-plane through $0$ that is perpendicular to $L$. Denote $D_r$ the open disk of radius $r$ centered at $0$ on the plane $P_\infty$.

	\begin{prop}\label{prop.tangent.cone.at.the.corner}
		The tangent cone $T_\infty\subset\RR^3$ is a single-sheeted planar domain that intersects $F_{j,\infty}$, $j=1,2$, at angle $\gamma_j$. Moreover, it is \textit{unique}. Namely, $T_\infty$ does not depend on the choice of subsequence for its construction.
	\end{prop}

	\begin{proof}
		We first notice that the tangent cone $\|T_\infty\|$ is nontrivial by virtue of \eqref{lower.density.bound}. Moreover, since $(T,E)$ solves the variational problem \eqref{capillary.variational.problem}, $(T_\infty,E_\infty)$ minimizes the corresponding energy in $\RR^3$. Precisely, let $C$ be the open set in $\RR^3$ enclosed by $F_{1,\infty}$ and $F_{2,\infty}$. Then for any open subset $U\subset\subset \RR^3$, $E_\infty$ minimizes the energy
		\begin{equation}\label{variation.in.r^3}
		\cF(E_\infty')=\cH^2(\partial E_\infty'\cap \mathring{C}\cap U)-\sum_{j=1}^2(\cos \gamma_j)\cH^2(\partial E_\infty'\cap \partial C\cap U).
		\end{equation}
		
		It follows immediately that $T_\infty$ is minimal in $\mathring{C}$. Therefore each connected component of $T_\infty$ is part of a two-plane. We conclude that
		\begin{equation}
		T_\infty=\bigcup_{j=1}^N \pi_j\cap C,
		\end{equation}
		where $\pi_j$ are planes through the origin and $\pi_i\cap\pi_j\cap C=\emptyset$ for $i\ne j$. Therefore we conclude either
		\begin{itemize}
			\item[Case 1] $N=1$ and $T_\infty=\pi_1\cap C$ for some plane $\pi_1$ such that $\pi_1\cap L_\infty=\{0\}$, or
			\item[Case 2] $N<\infty$ and $T_\infty=\cup_{j=1}^N \pi_j\cap C$, where $\pi_1,\cdots,\pi_N$ are planes with the line $L_\infty$ in common.
		\end{itemize}
		
		Now we rule out case 2 by constructing proper competitors. Notice that in case 2, $E_\infty=E^{(1)}_\infty\times \RR$ for some open $E^{(1)}_\infty\subset P_\infty$, here $P_\infty$ is a wedge region in $\RR^2$ such that $C=P_\infty\times \RR$, and $\partial E^{(1)}_\infty$ a finite union of rays emanating from the origin. Define the functional
		\begin{equation}\label{variation.in.r^2}
		\cF_\infty^{(1)}(E')=\cH^1(\partial E'\cap C\cap D_1)-\sum_{j=1}^2 (\cos\gamma_j) \cH^1(\partial E'\cap F_{j,\infty}\cap D_1).
		\end{equation}
		Notice that since $E_\infty$ minimizes \eqref{variation.in.r^3}, 
		\[\cF_\infty^{(1)}(E^{(1)}_\infty)\le \cF_\infty^{(1)}(E'),\]
		for any competitor $E'$.
		
		Observe that $P_\infty \setminus \overline{E^{(1)}_\infty}$ satisfies a variational principle similar to that satisfied by $E^{(1)}_\infty$ but with $\pi-\gamma_j$ in place of $\gamma_j$. In case $N>1$, we may simply ``smooth out" the vertex of $(\pi_1\cap \pi_2)\cap \overline{D_1}$ to decrease the functional $E^{(1)}_\infty$. Thus $N=1$. Without loss of generality assume that $\gamma_1\le \gamma_2$.
		
		To show that $N=1$ in case 2 cannot happen, we first observe that if $\beta_0$ is the angle formed by $E^{(1)}_\infty$ and $F_{1,\infty}$ at $0$, then $\beta_0\ge\gamma_1$. Otherwise we may construct a competitor $E'$ that has strictly smaller functional \eqref{variation.in.r^2} as follows. Let $q_1\in \partial D_{1/2}\cap (\partial E^{(1)}_\infty\setminus \partial C)$ and let $q_2$ be the point on $\partial E^{(1)}_\infty\cap F_{1,\infty}$ at distance $\ep$ from $0$. Then let $E'=E^{(1)}_\infty\setminus H$, where $H$ is the closed half plane with $0\in H\setminus \partial H$ and $\{q_1,q_2\}\in \partial H$. If $\beta_0<\gamma_1$, then it is easily checked (as illustrated in Figure \ref{pic.competitor}) that
		\[\cF_\infty^{(1)}(E')<\cF_\infty^{(1)}(E_\infty^{(1)}).\]
		
		\begin{figure}[htbp]
			\centering
			\includegraphics{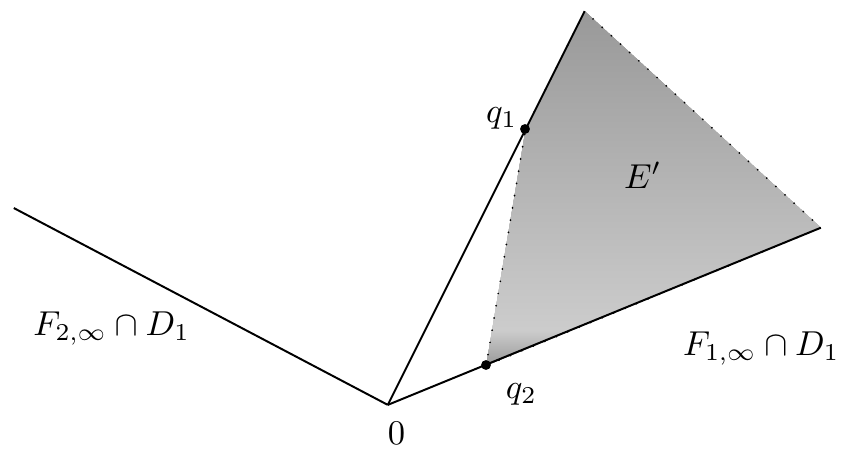}
			\caption{The construction of a competitor when $\beta_0<\gamma_1$.}
			\label{pic.competitor}
		\end{figure}
		
		On the other hand, since $P_\infty \setminus \overline{E^{(1)}_\infty}$ satisfies a similar variational principle with angle $\pi-\gamma_j$ in place of $\gamma_j$, we deduce that
		\[\theta-\beta_0\ge \pi-\gamma_2.\]
		We therefore conclude that $\theta\ge \pi+\gamma_1-\gamma_2$, contradiction. Thus case 2 is impossible.

		In case 1, $T_\infty$ contains a single sheet of plane. Namely, there exists some plane $\pi_1\subset \RR^3$ such that $T_\infty=\pi_1\cap C$. 
		
		Notice also that the plane $\pi_1\subset \RR^3$ should have constant contact angle along $F_{j,\infty}$, $j=1,2$:
		\begin{equation}\label{condition.for.pi1}
		\measuredangle(\pi_1,F_{1,\infty})=\gamma_1,\quad \measuredangle(\pi_1,F_{2,\infty})=\gamma_2,
		\end{equation}
		since everywhere on $\partial \Sigma\cap (F_j\setminus L)$, $\Sigma$ and $F_j$ meet at constant contact angle $\gamma_j$. We point out that the angle assumption \eqref{recall.angle.condition} is also a \textit{necessary and sufficient condition} for the existence of $\pi_1\subset \RR^3$. As a consequence, $T_\infty=\pi_1\cap \mathring{C}$ with $\pi_1$ uniquely determined by \eqref{condition.for.pi1}, independent of choice of the particular sequence of $r_k$ chosen to construct $T_\infty$. In other words, we have the strong property that the tangent cone is \textit{unique} for $T$ at $q$.
	\end{proof}
	
	\begin{rema}
		This part of the argument relies heavily on the fact that $T$ is two dimensional in two ways:
		\begin{itemize}
			\item The planes are the only minimal cones in $\RR^3$.
			\item A plane is uniquely determined by its intersection angles with two fixed planes.
		\end{itemize}
		Neither of these two statements is valid in higher dimensions.
	\end{rema}
	
	\begin{rema}
		The proof suggests that without the upper bound $\theta<\pi-|\gamma_1-\gamma_2|$, the tangent cone of $T$ at the corners could be a half plane through $L_\infty$. Moreover, $T_\infty$ may depend on the choice of the sequences of $r_k$ in its construction.
	\end{rema}
	
	\subsubsection{Curvature estimates and consequences}
	We prove a curvature estimate for $\Sigma$ near the corner $q$. Combined with the uniqueness of tangent cone, we deduce that $\Sigma$ must be graphical over its tangent plane at $q$. Then we may apply the PDE theory from \cite{Lieberman88Holder} to conclude that $\overline{\Sigma}$ is a $C^{1,\alpha}$ surface.
	
	We begin with the following lemma, which is a consequence of the monotonicity formula.
	
	\begin{lemm}\label{lemma.stable.surfaces.in.a.wedge}
		Let $C\in \RR^3$ be an open subset enclosed by two planes $F_1,F_2$ with $\measuredangle(F_1,F_2)=\theta$. Let $\Sigma$ be an area minimizing surface in $C$ such that $\Sigma$ intersects $F_j$ at constant angles $\gamma_1,\gamma_2$, and that $\cH^2(\Sigma\cap B_0(R))<C_0 R^2$ holds for large $R$ and some $C_0>0$. Assume also that
		\[|\pi-(\gamma_1+\gamma_2)|<\theta<\pi-|\gamma_1-\gamma_2|.\]
		Then there is a plane $\pi_1\subset \RR^3$ such that $\Sigma=\pi_1\cap C$.
	\end{lemm}
	
	\begin{proof}
		Without loss of generality assume $0\in\overline{\Sigma}$. Consider the tangent cone of $\Sigma$ at $\infty$ and at $0$. Since $\Sigma$ satisfies the angle assumption \eqref{recall.angle.condition}, its tangent cone at $0$, denoted by $\Sigma_0$, exists and is planar. Now by the monotonicity formula \eqref{monotonicity.for.corners} and the growth assumption $\cH^2(\Sigma\cap B_0(R))<C_0 R^2$, its tangent cone at infinity, denoted by $\Sigma_\infty$, exists and is an area minimizing cone. Since $\Sigma_0$ and $\Sigma_\infty$ are both minimal cones in $C\subset \RR^3$, they are parts of planes. However, the same argument as in the proof of Proposition \ref{prop.tangent.cone.at.the.corner} implies that $\Sigma_0=\Sigma_\infty=\pi\cap C$, where $\pi$ is the unique plane intersecting $F_j$ at angle $\gamma_j$. Therefore the equality in the monotonicity formula holds, and $\Sigma=\Sigma_0=\Sigma_\infty$ is planar.
	\end{proof}
	
	We are ready to prove the curvature estimates:
	
	\begin{prop}
		Let $\Sigma=\spt(T)\cap \mathring{M}$ be a minimizer of the variational problem \eqref{capillary.variational.problem}. Let $L=F_1\cap F_2$, $q\in \partial\Sigma\cap L$. Then the following curvature estimate holds:
		\begin{equation}\label{curvature.estimate.at.the.corner}
		|A_\Sigma|(x)\cdot \dist(x,q)\rightarrow 0,
		\end{equation}
		as $x\in \Sigma$ converges to $q$.
	\end{prop}

	\begin{proof}
		Assume, for the sake of contradiction, that there is $\delta>0$ and a sequence of points $q_k\in \Sigma$ such that
		\[\dist(q_k,q)=\ep_k>0, \quad \ep_k|A_\Sigma|(q_k)=\delta_k>\delta.\]
		By a standard point-picking argument, we could also assume that 
		\begin{equation}\label{point.picking}
		|A_\Sigma|(x)<\frac{2\delta_k}{\ep_k}, \quad x\in C_{2\ep_k}.
		\end{equation}
		Consider the rescaled surfaces
		\[\Sigma_k=\frac{\delta_k}{\ep_k}(\Sigma-q_k)\subset \frac{\delta_k}{\ep_k}(M-q_k).\]
		Since $\delta_k>\delta$, $\ep_k\rightarrow 0$, the ambient manifold $M$ converges, in the sense of Gromov-Hausdorff, to $(C,g_{Euclid})$. Since $\Sigma_k$ is area minimizing, a subsequence (which we still denote by $\Sigma_k$) converges to an area minimizing surface $\Sigma_\infty$. By \eqref{lower.density.bound}, $\Sigma_\infty$ is nontrivial. We consider two different cases
		\begin{itemize}
			\item If $\limsup_{k}\delta_k=\infty$, then by taking a further subsequence (which we still denote by $\Sigma_k$), $\Sigma_k$ converges to an area minimizing surface in $\RR^3$. Moreover, \eqref{point.picking} guarantees that the $|A_{\Sigma_\infty}|(x)<2$ for all $x\in \RR^3$. Therefore the convergence $\Sigma_k\rightarrow \Sigma_\infty$ is, in fact, in $C^\infty$. This produces a contradiction, since $|A_{\Sigma_k}|(0)=1$ for all $k$, and $\Sigma_\infty$ is a plane through the origin.
			\item If $\limsup_k \delta_k<C<\infty$, then the sequence $\Sigma_k$ converges to an area minimizing surface in the open set $C\subset \RR^3$ enclosed by the two limit planes. This produces a similar contradiction, because $|A_{\Sigma_k}|(0)=1$, $\Sigma_k\rightarrow \Sigma_\infty$ smoothly, and by Lemma \ref{lemma.stable.surfaces.in.a.wedge}, $\Sigma_\infty$ is flat in its interior.
		\end{itemize}
	\end{proof}
	
	With the curvature estimate, we may conclude the regularity discussion by concluding that $\Sigma$ is graphical near the corner $q$:
	
	\begin{prop}\label{prop.graphical.near.corner}
		Let $\Sigma$ be an energy minimizer of \eqref{capillary.variational.problem}, $q\in \overline{\Sigma}\cap L$. Then $\Sigma$ is a graph over the tangent plane at $q$, and its normal vector extends H\"older continuously to $q$; thus $\Sigma$ is a $C^{1,\alpha}$ surface with corners.
	\end{prop}
	
	\begin{proof}
		We first prove that $\Sigma$ is graphical near $q$. Embed a neighborhood of $q$ isometrically into some Euclidean space $\RR^N$. Take the unique plane $\pi_1\subset T_q M$ obtained above such that the tangent cone of $\Sigma$ at $q$ is $\pi_1\cap C$. Assume, for the sake of contradiction, that there is a sequence of points $q_k\in \Sigma$, $\dist_M(q_k,q)\rightarrow 0$, and that the normal vectors $N_k$ of $\Sigma\subset M$ at $q_k$ is parallel to $\pi_1$. Denote $\ep_k=\dist_M(q_k,q)$. Consider the rescaled surfaces $\Sigma_k=\ep_k^{-1}(\Sigma-q)$. By the monotonicity formula \eqref{monotonicity.for.corners} and the lower density bound \eqref{lower.density.bound}, a subsequence of $\{\Sigma_k\}$ converges to the unique tangent cone $\pi_1\cap C$ in the sense of varifolds. Notice that on $\Sigma_k$, the image of $q_k$ under the homothety has unit distance to the origin. By taking a further subsequence (which we still denote by $\{(\Sigma_k,q_k,N_k)\}$), we may assume that $q_k\rightarrow q_\infty$, $N_k\rightarrow N_\infty$, and $\dist_{\RR^3}(q_\infty,0)=1$. Now the curvature estimate \eqref{curvature.estimate.at.the.corner} implies that, 
		\[|A_{\Sigma_\infty}|(x)<2,\quad \text{for all points $x\in \Sigma\cap B_{1/2}(q_\infty)$.}\]
		For any point $x\in \Sigma_\infty$, and any curve $l$ connecting $q_\infty$ and $x$, we have
		\[|N_\infty(x)-N_\infty(q_\infty)|<\int_l |A_{\Sigma_\infty}|(y)d y.\]
		Therefore we conclude that, for points $x$ on a neighborhood $V$ of $q_\infty$ on $\Sigma_\infty$,
		\[|N_\infty(x)-N_{\pi_1}|>\frac{1}{2},\]
		where $N_{\pi_1}$ is the unit normal vector of $\pi_1$. This contradicts the fact that $\Sigma_k$ converges to $\Sigma_\infty$ as varifolds.
		
		Once we know that $\Sigma$ is a graph over $T_q\Sigma$ near $q$, the result of \cite{Lieberman88Holder} directly applies, and we conclude that $\Sigma$ has a H\"older continuous unit normal vector field up to $q$.
	\end{proof}

	\section{Non-rigid case}
	We prove Theorem \ref{theorem.comparison.nonrigid} in this section. Let $P$ be a polyhedron in $\RR^3$ of cone or prism types. Assume, for the sake of contradiction, that there exists a $P$-type polyhedron $(M^3,g)$ with $R(g)\ge 0$, $\overline{H}\ge 0$ and $\measuredangle_{ij}(M)<\measuredangle_{ij}(P)$. The strategy is to take the minimizer $\Sigma=\partial E$ of the \eqref{capillary.variational.problem}. When $M$ is of prism type, the existence and regularity of $\Sigma$ follows from the maximum principle in Proposition \ref{boundary.maximum.principle}. When $M$ is of cone type, we need the extra assumption that $\cI<0$ to guarantee that $E\ne \emptyset$. Hence we prove the following:
	
	\begin{lemm}\label{lemma.I.is.negative}
		Let $P\subset \RR^3$ be polyhedron of cone type, $(M,g)$ be of $P$-type. Assume $\measuredangle_{ij}(M)<\measuredangle_{ij}(P)$, then the infimum $\cI$ appeared in \eqref{capillary.variational.problem} is negative.
	\end{lemm}
	
	\begin{proof}
		As before let $F_j$, $F_j'$ denote the side faces of $M$, $P$, respectively; $B$, $B'$ denote their base faces. Assume, for the sake of contradiction, that 
		\begin{equation}\label{proof.lemma.I.is.negative}
		\cH^2(\partial E\cap \mathring{M})-\sum_{j=1}^k(\cos \gamma_j)\cH^2(\partial E\cap F_j)\ge 0.
		\end{equation}
		
		Notice that the inequality \eqref{proof.lemma.I.is.negative} is scaling invariant. Precisely, if $E\subset M$ satisfies \eqref{proof.lemma.I.is.negative}, then under the homothety $\mu_r$ defined by $x\mapsto r(x-p)$, the set $(\mu_r)_{\#}(E)\subset (\mu_r)_{\#}(M)$ satisfies \eqref{proof.lemma.I.is.negative}. Letting $r\rightarrow \infty$, the tangent cone $T_p M$of $M$ at $p$ should share the same property. Let $F_{j,\infty}$ denote the corresponding faces in $T_p M$. By assumption, $\measuredangle(F_{j,\infty},F_{j+1,\infty})<\measuredangle(F_j',F_{j+1}')$. Therefore $T_p M$ can be placed strictly inside the tangent cone of $P$ at its vertex. By elementary Euclidean geometry, there exists a plane $\pi\subset \RR^3$ such that $\pi$ meets $F_{j,\infty}$ with angle $\gamma_j'>\gamma_j$. See Figure \ref{pic.euclidean.lemma.1} for an illustration, where the dashed polyhedral cone is $T_p M$.
		
		\begin{figure}[htbp]
			\centering
			\includegraphics{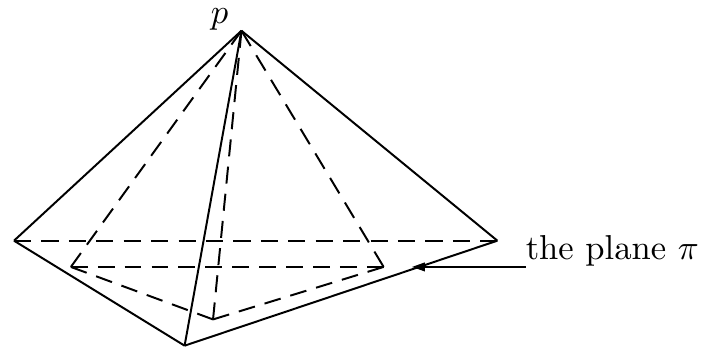}
			\caption{The tangent cone $T_p M$ contained in $P$.}
			\label{pic.euclidean.lemma.1}
		\end{figure}
		
		Let $\proj_{\pi}$ denote the projection $\RR^3\rightarrow \pi$. Then the Jacobian of $\proj_{\pi}$, restricted to each $F_{j,\infty}$, is $\cos \gamma_j'$. Denote $E_\infty$ the open domain enclosed by $\pi$ and $F_{j,\infty}$, $j=1,\cdots,k$. By the area formula,
		\[\cH^2(\pi\cap \partial E_\infty)-\sum_{j=1}^k(\cos \gamma_j')\cH^2(F_{j,\infty}\cap\partial E_\infty)=0.\]
		Since $\gamma_j'>\gamma_j$, we conclude
		\[\cH^2(\pi\cap \partial E_\infty)-\sum_{j=1}^k(\cos \gamma_j)\cH^2(F_{j,\infty}\cap\partial E_\infty)<0,\]
		contradiction.
	\end{proof}
	
	In the proof we are going to need another simple fact from Euclidean geometry.
	
	\begin{lemm}\label{lemma.Euclidean.geometry}
		Let $P_i,Q_i,R_i$, $i=1,2$, be six planes in $\RR^3$ with the property that $\measuredangle(P_1,R_1)=\measuredangle(P_2,R_2)$, $\measuredangle(Q_1,R_1)=\measuredangle(Q_2,R_2)$ and $\measuredangle(P_1,Q_1)\le \measuredangle(P_2,Q_2)$. Let $L_i=P_i\cap R_i$, $L'_i=Q_i\cap R_i$, $i=1,2$. Then $\measuredangle(L_1,L'_1)\le \measuredangle(L_2,L'_2)$.
	\end{lemm}
	
	\begin{proof}
		The proof is very similar to that of Lemma \ref{lemma.Euclidean.angle.condition}. Let $\uu_i,\vv_i,\ww_i$ be the unit normal vectors of $P_i,Q_i,R_i$, $i=1,2$, with the same choice of orientation, and such that $\uu_i,\vv_i$ are pointing out of the wedge region formed by $P_i$ and $Q_i$. Denote $\gamma_1=\measuredangle(P_1,R_1)=\measuredangle(P_2,R_2)$, $\gamma_2=\measuredangle(Q_1,R_1)=\measuredangle(Q_2,R_2)$, $\theta=\measuredangle(P_1,Q_1)$, and $\theta'=\measuredangle(P_2,Q_2)$. Then $\measuredangle(L_i,L_i')$ is given by the angle between $\uu_i\times \ww_i$ and $\ww_i\times \vv_i$, $i=1,2$. Notice that for $i=1,2$, $|\uu_i\times \ww_i|=\sin\gamma_1$, and $|\ww_i\times \vv_i|=\sin\gamma_2$. By assumptions, we have that
		\[\uu_1\cdot \vv_1=-\cos\theta,\quad \uu_2\cdot \vv_2=-\cos\theta',\quad \uu_i\cdot \ww_i=\cos \gamma_1,\quad \vv_i\cdot \ww_i=\cos\gamma_2.\]
		
		Therefore
		\[(\uu_1\times \ww_1)\cdot (\ww_1\times \vv_1)=\cos\gamma_1\cos\gamma_2+\cos\theta'.\]
		Hence if $\theta'\le \theta$, $(\uu_1\times \ww_1)\cdot (\ww_1\times \vv_1)\ge (\uu_2\times \ww_2)\cdot (\ww_2\times \vv_2)$. This implies that $\measuredangle(L_1,L'_1)\le \measuredangle(L_2,L'_2)$.
	\end{proof}
	
	Now we prove Theorem \ref{theorem.comparison.nonrigid}.
	
	\begin{proof}
		Assume, for the sake of contradiction, that $\measuredangle_{ij}(M)<\measuredangle_{ij}(P)$. By Theorem \ref{theorem.existence.regularity} and Lemma \ref{lemma.I.is.negative}, the infimum in \eqref{capillary.variational.problem} is achieved by an open set $E$, with $\Sigma=\partial E\cap \mathring{M}$ a smooth surface which is $C^{1,\alpha}$ up to its corners for some $\alpha\in (0,1)$. By the first variation formula \eqref{first.variation.for.regular.surfaces}, $\Sigma$ is capillary minimal. We apply the second variational formula \eqref{second.variation.formula} and conclude
		\begin{equation}
		\int_{\Sigma}[|\nabla f|^2-(|A|^2+\Ric(N,N))f^2]d\cH^2-\int_{\partial \Sigma}Q f^2 d\cH^1\ge 0,
		\end{equation}
		for any $C^2$ function $f$ compactly supported away from the corners, where on $\partial\Sigma\cap F_j$, 
		\[Q=\frac{1}{\sin \gamma_j}\secondfund(\overline{\nu},\overline{\nu})+(\cot \gamma_j) A(\nu,\nu).\]
		
		Since the surface $\Sigma$ is $C^{1,\alpha}$ to its corners, its curvature $|A|$ is square integrable. Hence by a standard approximation argument we conclude that the above inequality holds for the constant function $f=1$. We have
		\begin{multline}\label{second.variation.with.1}
		-\int_\Sigma (|A|^2+\Ric(N,N))-\\
		\sum_{j=1}^n\int_{\partial \Sigma\cap F_j}\left[\frac{1}{\sin \gamma_j}\secondfund(\overline{\nu},\overline{\nu})+\cot\gamma_j A(\nu,\nu)\right]\ge 0.
		\end{multline}
		
		Applying the Gauss equation on $\Sigma$, we have
		\begin{equation}\label{Gauss.equation}
		|A|^2+\Ric(N,N)=\frac{1}{2}(R-2K_\Sigma+|A|^2),    
		\end{equation}
		where $R$ is the scalar curvature of $M$, $K_\Sigma$ is the Gauss curvature of $\Sigma$. 
		
		By the Gauss-Bonnet formula for $C^{1,\alpha}$ surfaces with piecewise smooth boundary components, we have that
		\begin{equation}\label{Gauss.bonnet}
		\int_\Sigma K_\Sigma d\cH^2+\int_{\partial \Sigma}k_g d\cH^1+\sum_{j=1}^n(\pi-\alpha_j)= 2\pi\chi(\Sigma)\le 2\pi,
		\end{equation}
		here $k_g$ is the geodesic curvature of $\partial \Sigma\subset \Sigma$, and $\alpha_j$ are the interior angles of $\Sigma$ at the corners. By Lemma \ref{lemma.Euclidean.geometry}, $\alpha_j<\alpha_j'$, where $\alpha_j'$ is the corresponding interior angle of the base face of the Euclidean polyhedron $P$. Since $\sum_{j=1}^k (\pi-\alpha_j')=2\pi$, we conclude $\sum_{j=1}^k(\pi-\alpha_j)>2\pi$. As a result, we have that
		\begin{equation}\label{corollary.of.gauss.bonnet}
		-\int_\Sigma K_\Sigma d\cH^2> \int_{\partial \Sigma}k_g d\cH^1.
		\end{equation}
		
		Combining \eqref{second.variation.with.1}, \eqref{Gauss.equation} and \eqref{corollary.of.gauss.bonnet} we conclude that
		\begin{multline}\label{eq.to.simplify}
		\int_\Sigma\frac{1}{2}\left(R+|A|^2\right)d\cH^2\\
		+\sum_{j=1}^n\int_{\partial \Sigma\cap F_j}\left[\frac{1}{\sin \gamma_j}\secondfund(\overline{\nu},\overline{\nu})+\cot\gamma_j A(\nu,\nu)+k_g \right]d\cH^1< 0.
		\end{multline}
		
		To finish the proof, let us analyze the last integrand in \eqref{eq.to.simplify}. Fix one $j$ and consider $\partial \Sigma\cap F_j$. For convenience let $\gamma=\gamma_j$. We make the following
		
		\begin{claim}
			\begin{equation}\label{boundary.geometric.equation}
			\secondfund(\overline{\nu},\overline{\nu})+\cos \gamma A(\nu,\nu)+\sin \gamma k_g=\overline{H},
			\end{equation}
			where $\overline{H}$ is the mean curvature of $\partial M$ in $M$.
		\end{claim}
		Let $T$ be the unit tangential vector of $\partial \Sigma$. Since $\Sigma$ is minimal, $A(\nu,\nu)=-A(T,T)$. Therefore 
		\begin{align*}
		\cos\gamma A(\nu,\nu)+\sin\gamma k_g&=-\cos\gamma A(T,T)+\sin\gamma k_g\\
		&=-\bangle{\nabla_T T,\cos\gamma N+\sin\gamma \nu}\\
		&=-\bangle{\nabla_T T,X}\\
		&=\secondfund(T,T).
		\end{align*}
		Since $T$ and $\overline{\nu}$ form an orthonormal basis of $\partial M$, we have
		\begin{equation*}
		\secondfund(\overline{\nu},\overline{\nu})+\cos \gamma A(\nu,\nu)+\sin \gamma k_g=\secondfund(T,T)+\secondfund(\overline{\nu},\overline{\nu})=\overline{H}.
		\end{equation*}
		The claim is proved.
		
		To finish the proof, we note that \eqref{eq.to.simplify} implies that
		\begin{equation}
		\int_\Sigma\frac{1}{2}\left(R+|A|^2\right)d\cH^2+\sum_{j=1}^n\int_{\partial \Sigma\cap F_j}\frac{1}{\sin \gamma_j}\overline{H} d\cH^1< 0,
		\end{equation}
		contradicting the fact that the scalar curvature $R$ of $M$ and the surface mean curvature $\overline{H}$ of $\partial M\subset M$ are nonnegative.
	\end{proof}

	\section{Rigidity}
	In this section we prove Theorem \ref{theorem.rigidity}. Rigidity properties of minimal and area-minimizing surfaces have attracted lots of interests in recent years. Following the Schoen-Yau proof of the positive mass theorem, Cai-Galloway \cite{CaiGalloway00rigidity} studied the rigidity of area-minimizing tori in three-manifolds in nonnegative scalar curvature. The case of area-minimizing spheres was carried out by Bray-Brendle-Neves \cite{BrayBrendleNeves10rigidity}. Their idea is to study constant mean curvature (CMC) foliation around an infinitesimally rigid area-minimizing surface, and obtain a local splitting result for the manifold. It is very robust and applies to a wide variety of rigidity analysis: in the case of negative \cite{Nunes13rigidity} scalar curvature, and for area-minimizing surfaces with boundary \cite{Ambrozio15rigidity} (see also \cite{MicallefMoraru15splitting}). We adapt their idea for our rigidity analysis, and perform a dynamical analysis for foliations with constant mean curvature capillary surfaces. The new challenge here is that, when $M$ is of cube type, the energy minimizer of \eqref{capillary.variational.problem} may be empty. In this case the tangent cone $T_p M$ coincides with that of the Euclidean model $P$, and $\cI=0$. Our strategy, motivated by the earlier work of Ye \cite{ye1991foliation}, is to construct a constant mean curvature foliation near the vertex $p$, such that the mean curvature on each leaf converges to zero when approaching $p$.
	
	\subsection{Infinitesimally rigid minimal capillary surfaces}
	Assume the energy minimizer $\Sigma=\partial E\cap \mathring{M}$ exists for \eqref{capillary.variational.problem}. Tracing equality in the proof in Section 3, we conclude that
	\begin{equation}\label{rigidity.equation.1}
	\begin{gathered}
	\chi(\Sigma)=0,\quad R_M=0,\quad |A|=0\quad \text{on $\Sigma$} \\
	\overline{H}=0\quad \text{on $\partial \Sigma$},\qquad \alpha_j=\alpha_j'\quad \text{at the corners of $\Sigma$}.
	\end{gathered}
	\end{equation}
	Moreover, by the second variation formula \eqref{second.variation.formula},
	\begin{multline*}
	Q(f,f)=-\int_{\Sigma}(f\Delta f+(|A|^2+\Ric(N,N))f^2)d\cH^{n-1}\\
	+\sum_{j=1}^k\int_{\partial\Sigma\cap F_j}f\left(\pa{f}{\nu}-Q f \right)d\cH^{n-2}\ge 0,
	\end{multline*}
	with $Q(1,1)=0$. We then conclude that for any $C^2$ function $f$ compactly supported away from the vertices of $\Sigma$, $Q(1,f)=0$. By choosing appropriate $g$, we further conclude that
	\begin{equation*}
	\Ric(N,N)=0\quad \text{on $\Sigma$},\quad \frac{1}{\sin\gamma_j}\secondfund(\overline{\nu},\overline{\nu})+\cot\gamma_j A(\nu,\nu)=0\quad \text{on $\partial\Sigma \cap F_j$}.
	\end{equation*}
	Combining with \eqref{Gauss.equation} and \eqref{boundary.geometric.equation}, we conclude that
	\begin{equation}\label{rigidity.equation.2}
	K_\Sigma=0 \quad \text{on $\Sigma$},\quad k_g=0\quad\text{on $\partial\Sigma$}.
	\end{equation}
	
	Call a surface $\Sigma$ satisfying \eqref{rigidity.equation.1} and \eqref{rigidity.equation.2} \textit{infinitesimally rigid}. Notice that such a surface is isometric to an flat $k$-polygon in $\RR^2$.
	
	Next, we construct a local foliation by CMC capillary surfaces $\Sigma_t$. Take a vector field $Y$ defined in a neighborhood of $\Sigma$, such that $Y$ is tangential when restricted to $\partial M$. Let $\phi(x,t)$ be the flow of $Y$. Precisely, we have:
	
	\begin{prop}\label{prop.foliation.around.a.minimizer}
		Let $\Sigma^2$ be a properly embedded, two-sided, minimal capillary surface in $M^3$. If $\Sigma$ is infinitesimally rigid, then there exists $\ep>0$ and a function $w:\Sigma\times (-\ep,\ep)\rightarrow\RR$ such that, for every $t\in (-\ep,\ep)$, the set
		\[\Sigma_t=\{\phi(x,w(x,t):x\in\Sigma)\}\]
		is a capillary surface with constant mean curvature $H(t)$ that meets $F_j$ at constant angle $\gamma_j$. Moreover, for every $x\in \Sigma$ and every $t\in (-\ep,\ep)$,
		\[w(x,0)=0,\quad \int_\Sigma(w(x,t)-t)d\cH^2=0 \quad \text{and}\quad\pa{}{t}w(x,t)\bigg\vert_{t=0}=1.\]
		Thus, by possibly choosing a smaller $\ep$, $\{\Sigma_t\}_{t\in(-\ep,\ep)}$ is a foliation of a neighborhood of $\Sigma_0=\Sigma$ in $M$.
	\end{prop}
	
	Our proof goes by an argument involving the inverse function theorem, and is essentially taken from \cite{BrayBrendleNeves10rigidity} and \cite{Ambrozio15rigidity}. We do, however, need an elliptic theory on cornered domains. This is done by Lieberman \cite{Lieberman89optimal}. The following Schauder estimate is what we need:
	
	\begin{theo}[Lieberman,\cite{Lieberman89optimal}]\label{theorem.elliptic.equation.over.cornered.domains}
		Let $\Sigma^2\subset \RR^3$ be an open polygon with interior angles less than $\pi$. Let $L_1,\cdots,L_k$ be the edges of $\Sigma$. Then there exists some $\alpha>0$ depending only on the interior angles of $\Sigma$, such that if $f\in C^{0,\alpha}(\overline{\Sigma})$, $g\vert_{\overline{L_j}}\in C^{0,\alpha}(\overline{L_j})$, then the Neumann boundary problem
		\begin{equation}
		\begin{cases}
		\Delta u=f\quad \text{in $\Sigma$}\\
		\pa{u}{\nu}=g\quad \text{on $\partial \Sigma$}
		\end{cases}
		\end{equation}
		has a solution $u$ with $\int_\Sigma u=0$, and $u\in C^{2,\alpha}(\Sigma)\cap C^{1,\alpha}(\overline{\Sigma})$. Moreover, the Schauder estimate holds:
		\[|u|_{2,\alpha,\Sigma}+|u|_{1,\alpha,\overline{\Sigma}}\le C(|f|_{0,\alpha,\overline{\Sigma}}+
		\sum_{j=1}^k |g|_{0,\alpha,L_j}).\]
	\end{theo}
	
	We now prove Proposition \ref{prop.foliation.around.a.minimizer}.
	
	\begin{proof}
		For a function $u\in C^{2,\alpha}(\Sigma)\cap C^{1,\alpha}(\overline{\Sigma})$, consider the surface $\Sigma_u=\{\phi(x,u(x)):x\in \Sigma\}$, which is properly embedded if $|u|_0$ is small enough. We use the subscript $u$ to denote the quantities associated to $\Sigma_u$. For instance, $H_u$ denotes the mean curvature of $\Sigma_u$, $N_u$ denotes the unit normal vector field of $\Sigma_u$, and $X_u$ denotes the restriction of $X$ onto $\Sigma_u$. Then $\Sigma_0=\Sigma$, $H_0=0$, and $\bangle{N_u,X_u}=\cos\gamma_j$ along $\partial\Sigma\cap F_j$.
		
		Consider the Banach spaces
		\begin{equation*}
		\begin{gathered}
		F=\left\{u\in C^{2,\alpha}(\Sigma)\cap C^{1,\alpha}(\overline{\Sigma}):\int_\Sigma u=0\right\},\\
		G=\left\{u\in C^{0,\alpha}(\Sigma):\int_\Sigma u=0\right\},\quad H=\left\{u\in L^{\infty}(\partial \Sigma): u\vert_{\overline{L_j}} \in C^{0,\alpha}(\overline{L_j})\right\}.
		\end{gathered}
		\end{equation*}
		Given small $\delta>0$ and $\ep>0$, define the map $\Psi:(-\ep,\ep)\times (B_0(\delta)\subset F)\rightarrow G\times H$ given by
		\[\Psi(t,u)=\left(H_{t+u}-\frac{1}{|\Sigma|}\int_\Sigma H_{t+u},\bangle{N_{t+u},X_{t+u}}-\cos\gamma\right),\]
		where $\gamma=\gamma_j$ on $\partial\Sigma\cap \mathring{F_j}$.
		
		In order to apply the inverse function theorem, we need to prove that $D_u\Psi\vert_{(0,0)}$ is an isomorphism when restricted to $\{0\}\times F$. In fact, for any $v\in F$, 
		\[D_u\Psi\vert_{(0,0)}(0,v)=\td{}{s}\bigg\vert_{s=0}\psi(0,sv)=\left(\Delta v-\frac{1}{|\Sigma|}\int_{\partial \Sigma}\pa{v}{\nu},-\pa{v}{\nu}\right).\]
		
		The calculation is given in Lemma \ref{lemma.appendix.2} and Lemma \ref{lemma.appendix.3} in the appendix. Now the fact that $D_u\Psi\vert_{(0,0)}$ is an isomorphism follows from Theorem \ref{theorem.elliptic.equation.over.cornered.domains}. The rest of the proof is the same as Proposition 10 in \cite{Ambrozio15rigidity}, which we will omit here.
	\end{proof}

	\subsection{CMC capillary foliation near the vertex}
	When $(M^3,g)$ is of cone type with vertex $p$, we have proved that $\cI$ is realized by a minimizer $\partial E\ne\emptyset$ when $\cI<0$. Now it is obvious from the definition that $\cI\le 0$. However, in the case that $\cI=0$, it is a priori possible that the minimizer $E=\emptyset$. Assume $\cI=0$. We investigate this case with a different approach.
	
	Notice that, as a consequence of Lemma \ref{lemma.I.is.negative}, $\cI=0$ implies that
	\[\measuredangle(F_j,F_{j+1})\vert_p=\measuredangle(F_j',F_{j+1}'),\]
	where $F_j'$ is the corresponding face of the Euclidean model $P$. Recall that in the Euclidean model $P'$, its base face $B'$ intersects $F_j'$ at angle $\gamma_j$. Thus $P$ is foliated by a family of planes parallel to $B'$, where each leaf is minimal, and meets $F_j'$ at constant angle $\gamma_j$. We generalize this observation to arbitrary Riemannian polyhedra, and obtain:
	
	\begin{theo}\label{theorem.foliation.near.a.vertex}
		Let $(M^3,g)$ be a cone type Riemannian polyhedron with vertex $p$. Let $P\subset \RR^3$ be a polyhedron with vertex $p'$, such that the tangent cones $(T_p M,g_p)$ and $(T_{p'} P,g_{Euclid})$ are isometric. Denote $\gamma_1,\cdots,\gamma_k$ the angles between the base face and the side faces of $P$. Then there exists a small neighborhood $U$ of $p$ in $M$, such that $U$ is foliated by surfaces $\{\Sigma_\rho\}_{\rho\in (0,\ep)}$ with the properties that:
		\begin{enumerate}
			\item for each $\rho\in (0,\ep)$, $\Sigma_\rho$ meet the side face $F_j$ at constant angle $\gamma_j$;
			\item each $\Sigma_\rho$ has constant mean curvature $\lambda_\rho$, and $\lambda_\rho\rightarrow 0$ as $\rho\rightarrow 0$.
		\end{enumerate}
	\end{theo}
	
	\begin{rema}
		Before proceeding to the proof, let us remark that the local foliation structure of Riemannian manifolds has been a thematic program in geometric analysis, and has deep applications to mathematical general relativity. See: Ye \cite{ye1991foliation} for spherical foliations around a point; Huisken-Yau \cite{HuiskenYau96definition} for foliations in asymptotically flat spaces; Mahmoudi-Mazzeo-Pacard \cite{MazzeoPacard05foliations}\cite{MahmoudiMazzeoPacard06constant} for foliations around general minimal submanifolds.
	\end{rema}
	
	\begin{rema}
		As a technical remark, let us recall that in all of the aforementioned foliation results, some extra conditions are necessary (e.g. Ye's result required the center point to be a non-degenerate critical point of scalar curvature; Mahmoudi-Mazzeo-Pacard needed the minimal submanifold to be non-degenerate critical point for the volume functional). However, in our result, no extra condition is needed. Geometrically, this is because in the tangent cone $T_p M\subset\RR^3$, the desired foliation is \underline{unique}.
	\end{rema}
	
	\begin{proof}
		Let $U$ be a small neighborhood of $p$ in $M$. Take a local diffeomorphism $\varphi:P\rightarrow U$, such that the pull back metric $\varphi^*g$ and $g_{Euclid}$ are $C^1$ close. Place the vertex $p'$ of $P$ at the origin of $\RR^3$. In local coordinates on $\RR^3$, the above requirement is then equivalent to
		\[g_{ij}(0)=g_{ij,k}=0,\quad g_{ij}(x)=o(|x|), g_{ij,k}(x)=o(1) \text{ for $|x|<\rho_0$}.\]
		$\varphi$ may be constructed, for instance, via geodesic normal coordinates.
		
		Denote $\overline{M}\subset\RR^3$ the tangent cone of $M$ at $p$. By assumption, the dihedral angles $\measuredangle(F_i,F_j)\vert_p=\measuredangle(F_i',F_j')$. Let $\pi$ be the plane in $\RR^3$ such that in Euclidean metric, $\pi$ and $F_j$ meet at constant angle $\gamma_j$. For $\rho\in (0,1]$, let $\pi_\rho$ be the plane that is parallel to $\pi$ and has distance $\rho$ to $0$. Let $\Sigma_\rho$ be the intersection of $\pi_\rho$ with the interior of the cone $T_p M$. Denote $X$ the outward pointing unit normal vector field on $\partial \overline{M}$, $N_\rho$ the unit vector field of $\Sigma_\rho\subset \overline{M}$ pointing towards $0$. Denote $Y$ the vector field such that for each $x\in \Sigma_\rho$, $Y(x)$ is parallel to $\vec{x}$. Moreover, we require that the flow of $Y$ parallel translates $\{\Sigma_\rho\}$, and $Y(x)$ is tangent to $\partial \overline{M}$ when $x\in \partial\overline{M}$. Let $\phi(x,t)$ be the flow of $Y$. For a function $u\in C^{2,\alpha}(\Sigma_1)\cap C^{1,\alpha}(\overline{\Sigma_1})$ ($\Sigma_1$ is parallel to $\pi$, and of distance $1$ to the origin), define the perturbed surface
		\[\Sigma_{\rho,u}=\{\phi(\rho x,u(\rho x)):x\in \Sigma_1\}.\]
		Since $\Sigma_\rho=\rho\Sigma_1$, the surface $\Sigma_{\rho,u}$ is a small perturbation of $\Sigma_\rho$, and is properly embedded, if $|u|_0$ is sufficiently small.
		
		We use the subscript $\rho$ to denote geometric quantities related to $\Sigma_\rho$, and the subscript $(\rho,u)$ to denote geometric quantities related to the perturbed surfaces $\Sigma_{\rho,u}$, both in the metric $\varphi^*g$. In particular, $H_{\rho,u}$ denotes the mean curvature of $\Sigma_{\rho,u}$, and $N_{\rho,u}$ denotes the unit normal vector field of $\Sigma_{\rho,u}$ pointing towards $0$. It follows from Lemma \ref{lemma.appendix.1} and Lemma \ref{lemma.appendix.2} that we have the following Taylor expansion of geometric quantities.
		\begin{equation}\label{equation.perturbed.geometric.quantities}
		\begin{aligned}
		H_{\rho,u}&=H_\rho+\frac{1}{\rho^2}\Delta_\rho u+(\Ric(N_\rho,N_\rho)+|A_\rho|^2)u+L_1 u+Q_1(u)\\
		\bangle{X_{\rho,u},N_{\rho,u}}&=\bangle{X_\rho,u_\rho}-\frac{\sin\gamma_j}{\rho}\pa{u}{\nu_\rho}\\
		&\qquad+(\cos\gamma_j A(\nu_\rho,\nu_\rho)+\secondfund(\overline{\nu_\rho},\overline{\nu_\rho}))u+L_2 u+Q_2(u).
		\end{aligned}
		\end{equation}
		
		Let us explain \eqref{equation.perturbed.geometric.quantities} a bit more. $Q_1,Q_2$ are terms that are at least quadratic in $u$. The functions $L_1,L_2$ exhibit how the mean curvature $H_\rho$ and the contact angle $\gamma_j$ deviate from being constant. In particular, they are bounded in the following manner:
		\[L_1 \le C|\nabla_\rho H_\rho| |Y|\le C|g|_{C^2}<C,\quad L_2\le C|\nabla_\rho \bangle{X_\rho,N_\rho}| |Y|<C|g|_{C^1}<C.\]
		The operator $\Delta_\rho$ is the Laplace operator on $\Sigma_\rho$. At $x\in \Sigma_\rho$,
		\[\Delta_\rho= \frac{1}{\sqrt{\det(g)}} \partial_i \left(\sqrt{\det(g)}g^{ij}\partial_j\right).\]
		In particular, $\Delta_\rho$ converges to the Laplace operator on $\RR^2$ as $\rho\rightarrow 0$. In local coordinates, it is not hard to see that
		\[|H_\rho|\le C|g|_{C^1}=o(1),\quad |\bangle{X_\rho,N_\rho}-\cos\gamma_j|\le |g|_{C^0}=o(\rho).\]
		Denote $D_\rho=\bangle{X_\rho,N_\rho}-\cos\gamma_j$. Letting $H_{\rho,u}\equiv \lambda$, we deduce from \eqref{equation.perturbed.geometric.quantities} that we need to solve for $u$ from
		\begin{equation}\label{equation.for.u}
		\begin{cases}
		\Delta_\rho u+\rho^2L_1 u+\rho^2 Q_1(u)=\rho^2(\lambda-H_\rho) \quad\text{in $\Sigma_1$,}\\
		\pa{u}{\nu_\rho}=\rho D_\rho+\rho L_2u+\rho Q_2(u)\quad \text{on $\partial \Sigma_1$.}
		\end{cases}
		\end{equation}
		
		We use inverse function theorem as in the proof of Proposition \ref{prop.foliation.around.a.minimizer}. Precisely, denote the operator
		\[\begin{cases}\cL_\rho (u)=\Delta_\rho u-\rho^2 L_1u-\rho^2 Q_1(u)+\rho^2 H_\rho,\\ \cB_\rho(u)=\pa{u}{\nu_\rho}-\rho D_\rho-\rho L_2 u-\rho Q_2(u),\end{cases}\]
		and consider the Banach spaces
		\begin{equation*}
		\begin{gathered}
		F=\left\{u\in C^{2,\alpha}(\Sigma_1)\cap C^{1,\alpha}(\overline{\Sigma_1}):\int_{\Sigma_1} u=0\right\},\\
		G=\left\{u\in C^{0,\alpha}(\Sigma_1):\int_{\Sigma_1} u=0\right\}, H=\left\{u\in L^\infty(\partial \Sigma_1): u\vert_{\overline{L_j}} \in C^{0,\alpha}(\overline{L_j})\right\}.
		\end{gathered}
		\end{equation*}
		Again we use $L_1,\cdots,L_j$ to denote the edges of $\Sigma_1$.
		
		For a small $\delta>0$, let $\Psi:(-\ep,\ep)\times (B_\delta(0)\subset F)\rightarrow G\times H$ given by
		\[\Psi(\rho,u)=\left(\cL_\rho(u)-\frac{1}{|\Sigma_1|}\int_{\Sigma_1}\cL_\rho(u)d\cH^2,\cB_\rho(u)\right).\]
		By the asymptotic behavior as $\rho\rightarrow 0$ discussed above, the linearized operator $D_u\Psi\vert_{(0,0)}$, when restrited to $\{0\}\times F$, is given by
		\[D_u\Psi\vert_{(0,0)}(0,v)=\td{}{s}\bigg\vert_{s=0}\Psi(0,s v)=\left(\Delta v-\int_{\Sigma_1}\Delta v,\pa{v}{\nu}\right).\]
		By Theorem \ref{theorem.elliptic.equation.over.cornered.domains}, for some $\alpha\in (0,1)$, $D_u\Psi\vert_{(0,0)}$ is an isomorphism when restricted to $\{0\}\times F$. We therefore apply the inverse function theorem and conclude that, for small $\ep>0$, there exists a $C^1$ map between Banach spaces $\rho\in (-\ep,\ep)\mapsto u(\rho)\in B_\delta(0)\subset F$ for every $\rho\in (-\ep,\ep)$, such that $\Psi(\rho,u(\rho))=(0,0)$. Thus the surface $\Sigma_{\rho,u(\rho)}$ is minimal, and meets $F_j$ at constant angle $\gamma_j$.
		
		By definition, $u(0)$ is the zero function. Denote $v=\pa{u(\rho)}{\rho}$. Differentiating \eqref{equation.for.u} with respect to $\rho$ and evaluating at $\rho=0$, we deduce
		\begin{equation}
		\begin{cases}
		\Delta v=0 \quad\text{in $\Sigma_1$},\\
		\pa{v}{\nu}=0\quad\text{on $\partial \Sigma_1$}.
		\end{cases}
		\end{equation}
		Therefore $v$ is also the zero function. Thus we conclude that
		\[|u|_{1,\alpha,\overline{\Sigma_1}}=o(\rho),\]
		for $|\rho|<\rho_0$.
		
		Therefore the surfaces $\Sigma_{\rho,u(\rho)}$ is a foliation of a small neighborhood of $p$. Moreover, integrating \eqref{equation.for.u} over $\Sigma_1$, we find that the constant mean curvature of $\Sigma_{\rho,u(\rho)}$ satisfies
		\begin{equation}
		\begin{aligned}
		\lambda_\rho&=\frac{1}{\rho^2}\int_{\Sigma_1} \Delta u+ \int_{\Sigma_1}(L_1u+Q_1(u)+H_\rho)\\
		&=\frac{1}{\rho^2}\int_{\partial \Sigma_1} \pa{u}{\nu}+\int_{\Sigma_1}(L_1u+Q_1(u)+H_\rho)+o(1)\\
		&=\frac{1}{\rho}\int_{\partial \Sigma_1} (D_\rho+ L_2 u+Q_2(u))+\int_{\Sigma_1}(L_1u+Q_1(u)+H_\rho)+o(1).
		\end{aligned}
		\end{equation}
		
		Since
		\[D_\rho=o(\rho),\quad |u|_{1,\alpha,\overline{\Sigma_1}}=o(\rho),\quad H_\rho=o(1),\]
		we conclude that $\lambda_\rho\rightarrow 0$, as $\rho\rightarrow 0$.
	\end{proof}

	\subsection{Local splitting}
	We analyze the CMC capillary foliations developed above to prove a local splitting theorem, thus prove Theorem \ref{theorem.rigidity}. We need the extra assumption \eqref{extra.angle.assumption} that
	\[\gamma_j\le \pi/2, j=1,\cdots,k\quad \text{or}\quad\gamma_j\ge \pi/2,j=1,\cdots,k.\]
	First notice that, if $P\subset \RR^3$ is a cone, then \eqref{extra.angle.assumption} is possible only when $\gamma_j\le \pi/2$, $j=1,\cdots,k$; if $P$ is a prism and $\gamma_j>\pi/2$, then instead of \eqref{capillary.variational.problem}, we consider, for $E_1=M\setminus \overline{E}$,
	\begin{equation}
	\cF(E_1)=\cH^2(\partial E_1\cap \mathring{M})-\sum_{j=1}^k(\cos\gamma_j)\cH^2(\partial E_1\cap F_j),
	\end{equation}
	and reduce the problem to the case where $\gamma_j\le\pi/2$. Thus we always assume $\gamma_j\le \pi/2$, $j=1,\cdots,k$.
	
	Under the same conventions as before, assume we have a local CMC capillary foliation $\{\Sigma_\rho\}_{\rho\in I}$, where as $\rho$ increase, $\Sigma_\rho$ moves in the direction of $N_\rho$. We will take $I$ to be $(-\ep,\ep)$, $(-\ep,0)$ or $(0,\ep)$, according to the location of the foliation. We prove the following differential inequality for the mean curvature $H(\rho)$.
	
	\begin{prop}\label{proposition.differential.ineq.of.H}
		There exists a nonnegative continuous function $C(\rho)\ge 0$ such that
		\[H'(\rho)\ge C(\rho) H(\rho).\]
	\end{prop}
	
	\begin{proof}
		Let $\psi:\Sigma\times I\rightarrow M$ parametrizes the foliation. Denote $Y=\pa{\psi}{t}$. Let $v_\rho=\bangle{Y,N_\rho}$ be the lapse function. Then by Lemma \ref{lemma.appendix.1} and Lemma \ref{lemma.appendix.2}, we have
		\begin{align}
		&\td{}{\rho}H(\rho)=\Delta_\rho v_\rho+(\Ric(N_\rho,N_\rho)+|A_\rho|^2)v_\rho \quad\text{in $\Sigma_\rho$},\label{eq.derivative.of.mean.curvature}\\
		&\pa{v_\rho}{\nu_\rho}=\left[(\cot\gamma_j) A_\rho(\nu_\rho,\nu_\rho)+\frac{1}{\sin \gamma_j}\secondfund(\overline{\nu_\rho},\overline{\nu_\rho})\right]v_\rho \quad \text{on $\partial \Sigma_\rho\cap F_j$}.\label{eq.derivative.of.angle}
		\end{align}
		By shrinking the interval $I$ if possible, we may assume $v_\rho>0$ for $\rho\in I$. Multiplying $\frac{1}{v_\rho}$ on both sides of \eqref{eq.derivative.of.mean.curvature} and integrating on $\Sigma_\rho$, we deduce that
		\begin{equation}
		\begin{split}
		H&'(\rho)\int_{\Sigma_\rho}\frac{1}{v_\rho}=\int_{\Sigma_\rho}\frac{|\nabla v_\rho|^2}{v_\rho^2}d\cH^2+\frac{1}{2}\int_{\Sigma_\rho}(R+|A|^2+H^2)d\cH^2-\int_{\Sigma_\rho}K_{\Sigma_\rho}d\cH^2\\
		&\qquad\qquad\qquad+\sum_{j=1}^k\int_{\partial\Sigma_\rho\cap F_j}\left[\cot\gamma_j A_\rho(\nu_\rho,\nu_\rho)+\frac{1}{\sin\gamma_j}\secondfund(\overline{\nu_\rho},\overline{\nu_\rho})\right]d\cH^1\\
		&\ge -\int_{\Sigma_\rho}K_{\Sigma_\rho}d\cH^2+\sum_{j=1}^k\int_{\partial\Sigma_\rho\cap F_j}\left[\cot\gamma_j A_\rho(\nu_\rho,\nu_\rho)+\frac{1}{\sin\gamma_j}\secondfund(\overline{\nu_\rho},\overline{\nu_\rho})\right]d\cH^1.
		\end{split}
		\end{equation}
		Using the Gauss-Bonnet formula and Lemma \ref{lemma.Euclidean.geometry}, 
		\begin{equation}
		-\int_{\Sigma_\rho}K_{\Sigma_\rho}d\cH^2\ge \int_{\partial \Sigma_\rho}k_g d\cH^1.
		\end{equation}
		As in \eqref{boundary.geometric.equation}, we also have
		\begin{equation}
		k_g+\cot\gamma_j A(\nu_\rho,\nu_\rho)+\frac{1}{\sin\gamma_j}\secondfund(\overline{\nu_\rho},\overline{\nu_\rho})=(\cot \gamma_j) H(\rho)+\frac{1}{\sin \gamma_j}\overline{H},
		\end{equation}
		on $\partial\Sigma_\rho \cap F_j$. Combining these, we deduce
		
		\begin{equation}
		\begin{split}
		H'(\rho)\int_{\Sigma_\rho}\frac{1}{v_\rho}&\ge \sum_{j=1}^k\int_{\partial \Sigma_\rho\cap F_j}\left[(\cot\gamma_j) H(\rho)+\frac{1}{\sin\gamma_j}\overline{H}\right]d\cH^1\\
		&\ge \left[\sum_{j=1}^k(\cot\gamma_j)\cH^1(\partial \Sigma_\rho\cap F_j)\right]H(\rho).
		\end{split}
		\end{equation}
		Take $C(\rho)=\sum_{j=1}^k(\cot\gamma_j)\cH^1(\partial \Sigma_\rho\cap F_j)$. The proposition is proved.
		
	\end{proof}
	
	We are now ready to prove Theorem \ref{theorem.rigidity}.
	
	\begin{proof}
		If $(M^3,g)$ is of prism type, or if $(M^3,g)$ is of cone type with $\cI<0$, then the variational problem \eqref{capillary.variational.problem} has a nontrivial solution $E$ with a $C^{1,\alpha}$ boundary $\Sigma$. Therefore $\Sigma$ is infinitesimally rigid minimal capillary, and there is a CMC capillary foliation $\{\Sigma_\rho\}_{I}$ around $\Sigma$, where $I=(-\ep,\ep)$ if $\Sigma\subset \mathring{M}$, $I=[0,\ep)$ if $\Sigma=B_1$, and $I=(-\ep,0]$ if $\Sigma=B_2$. By Proposition \ref{proposition.differential.ineq.of.H}, the mean curvature $H(\rho)$ of $\Sigma_\rho$ satisfies
		\[\begin{cases}H(0)=0\\H'(\rho)\ge C(\rho) H(\rho),\end{cases}\]
		where $C(\rho)\ge 0$. By standard ordinary differential equation theory, 
		\[H(\rho)\ge 0 \text{ when $\rho\ge 0$},\quad H(\rho)\le 0 \text{ when $\rho\le 0$}.\]
		
		Denote $E_\rho$ the corresponding open domain in $M$. Since each $\Sigma_\rho$ meets $F_j$ at constant angle $\gamma_j$, the first variation formula \eqref{first.variation.for.regular.surfaces} implies that
		\[F(\rho_1)-F(\rho_2)=-\int_{\rho_2}^{\rho_1}d\rho\int_{\Sigma_\rho}H(\rho)v_\rho d\cH^2.\]
		We then conclude that for $\delta>0$,
		\[F(\delta)\le F(0),\qquad F(-\delta)\le F(0).\]
		However, $\Sigma_0=\Sigma$ minimizes the functional \eqref{capillary.variational.problem}. Therefore in a neighborhood of $\Sigma$, $F(\rho)=F(0)$, $H(\rho)\equiv 0$. Tracing back the equality conditions, we find that
		\[v_\rho\equiv \text{constant}, \quad \text{each $\Sigma_\rho$ is infinitesimally rigid}.\]
		
		It is then straightforward to check that the normal vector fields of $\Sigma_\rho$ is parallel (see \cite{BrayBrendleNeves10rigidity} or \cite{MicallefMoraru15splitting}). In particular, its flow is a flow by isometries and therefore provides the local splitting. Since $M$ is connected, this splitting is also global, and we conclude that $(M^3,g)$ is isometric to a flat polyhedron in $\RR^3$.
		
		If $(M^3,g)$ is of cone type with $\cI=0$, then by Theorem \ref{theorem.foliation.near.a.vertex}, there is a CMC capillary foliation $\{\Sigma_\rho\}_{\rho\in (-\ep,0)}$ near the vertex, with $H(\rho)\rightarrow 0$ as $\rho\rightarrow 0$. By Proposition \ref{proposition.differential.ineq.of.H}, the mean curvature $H(\rho)$ satisfies
		\[\begin{cases}H'(\rho)\ge C(\rho)H(\rho)\quad \rho\in (-\ep,0)\\H(\rho)\rightarrow 0\quad \rho\rightarrow 0.\end{cases}\]
		Since $C(\rho)\ge 0$, we conclude that $H(\rho)\le 0$, $\rho\in (-\ep,0)$. Let $E_\rho$ be the open subset bounded by $\Sigma_\rho$. Take $0<\eta<\delta$, then
		\[F(-\eta)-F(-\delta)=-\int_{-\delta}^{-\eta} d\rho \int_{\Sigma_\rho}H v_\rho d\cH^2\ge 0\quad\Rightarrow\quad F(-\delta)\le F(-\eta).\]
		Letting $\eta\rightarrow 0$, we have
		\[F(-\delta)\le 0.\]
		As before, we conclude that $F(\rho)\equiv 0$ for $\rho\in (-\ep,0)$, and that each leaf $\Sigma_\rho$ is infinitesimally rigid. Thus $(M^3,g)$ admits a global splitting of flat $k$-polygon in $\RR^2$, and hence is isometric to a flat polyhedron in $\RR^3$.
	\end{proof}
	
	\appendix
	\section{}
	We provide some general calculation for infinitesimal variations of geometric quantities of properly immersed hypersurfaces under variations of the ambient manifold $(M^{n+1},g)$ that leave the boundary of the hypersurface inside $\partial M$. We also refer the readers to the thorough treatment in \cite{RosSouam97capillarystability} and \cite{Ambrozio15rigidity} (warning: the choice of orientation for the unit normal vector field $N$ in \cite{Ambrozio15rigidity} is the opposite to ours).
	
	We keep the notations used in Section 2.1 and for each $t\in(-\ep,\ep)$, we use the subscript $t$ for the terms related to $\Sigma_t$. Recall that $Y=\pa{\Psi(t,\cdot)}{t}$ is the deformation vector field. Denote $Y_0$ the tangent part of $Y$ on $\Sigma$, $Y_0$ the tangent part of $Y$ on $\partial \Sigma$. Let $v=\bangle{Y,N}$. For $q\in \Sigma$, let $e_1,\cdots,e_n$ be an orthonormal basis of $T_q\Sigma$, and let $e_i(t)=d\Psi_t(e_i)$. Let $S_0,S_1$ be the shape operators of $\Sigma\subset M$ and $\partial M\subset M$. Precisely, $S_0(Z_1)=-\nabla_{Z_1}N$, $S_1(Z_2)=\nabla_{Z_2}X$. We have:
	
	\begin{lemm}[Lemma 4.1(1) of \cite{RosSouam97capillarystability}, Proposition 15 of \cite{Ambrozio15rigidity}]\label{lemma.appendix.1}
		\begin{equation}
		\nabla_Y N=-\nabla^\Sigma v-S_0(Y_0).
		\end{equation}
	\end{lemm}
	
	We use Lemma \ref{lemma.appendix.1} to calculate the evolution of the contact angle along the boundary.
	
	\begin{lemm}\label{lemma.appendix.2}
		Let $\gamma$ denote the contact angle between $\Sigma$ and $F_j$. Then
		\begin{equation}
		\td{}{t}\bigg\vert_{t=0}\bangle{N_t,X_t}=-\sin\gamma \pa{v}{\nu}+(\cos\gamma) A(\nu,\nu)v+\secondfund(\overline{\nu},\overline{\nu})v+\bangle{L,\nabla^{\partial\Sigma}\gamma_j}v,
		\end{equation}
		where $L$ is a bounded vector field on $\partial\Sigma$.
		
		In particular, if each $\Sigma_t$ meets $F_j$ at constant angle $\gamma_j$, then on $F_j$,
		\[\pa{v_t}{\nu_t}=\left[(\cot \gamma_j) A_t(\nu_t,\nu_t)+\frac{1}{\sin\gamma_j}\secondfund(\overline{\nu_t},\overline{\nu_t})\right]v_t.\]
	\end{lemm}
	
	\begin{proof}
		Let us fix one boundary face $F_j$ and denote $\gamma_j$ by $\gamma$. By Lemma \ref{lemma.appendix.1}, 
		\[\begin{split}
		\td{}{t}\bigg\vert_{t=0}\bangle{N_t,X_t}&=\bangle{\nabla_Y N,X}+\bangle{N,\nabla_Y X}\\
		&=-\bangle{\nabla^\Sigma v,X}-\bangle{S_0(Y_0),X}+\bangle{N,\nabla_Y X}.\end{split}\]
		On $\partial M$, $Y$ decomposes into $Y=Y_1-\frac{v}{\sin\gamma}\overline{\nu}$. Notice that since $X=\cos\gamma N+\sin\gamma N$, 
		\[\bangle{S_0(Y_0),X}=\bangle{S_0(Y_0),\cos\gamma N+\sin\gamma \nu}=\sin\gamma A(Y_0,\nu).\]
		
		We also have the vector decomposition on $\partial M$ with respect to the orthonormal basis $\overline{\nu},X$:
		\begin{equation}\label{appendix.vector.decomposition}
		N=\cos\gamma X-\sin\gamma \overline{\nu},\qquad \nu=\cos\gamma \overline{\nu}+\sin\gamma X.
		\end{equation}
		Since $\bangle{X,X}=1$ along $\partial M$, we have $\bangle{X,\nabla_Z X}=0$ for any vector $Z$ on $\partial M$. We have
		\[\begin{split}
		\td{}{t}\bigg\vert_{t=0}\bangle{N_t,X_t}&=-\sin\gamma\pa{v}{\nu}-\bangle{S_0(Y_0),X}\\
		&\qquad\qquad+\bangle{\cos\gamma X-\sin\gamma \overline{\nu},\nabla_{Y_1-\frac{v}{\sin\gamma}\overline{\nu}}X}\\
		&=-\sin\gamma\pa{v}{\nu}-\sin\gamma A(Y_0,\nu)-\sin\gamma \bangle{\overline{\nu},\nabla_{Y_1}X}+\bangle{\overline{\nu},\nabla_{\overline{\nu}}X}v.
		\end{split}\]
		
		Now we deal with the second and the third terms above. Notice that on $\partial \Sigma\cap F_j$,
		\[Y_0=Y_1-(\cot\gamma) v\nu.\]
		Thus $A(Y_0,\nu)=A(Y_1,\nu)-(\cot\gamma) v A(\nu,\nu)=-\bangle{\nabla_{Y_1}N,\nu}-(\cot \gamma) A(\nu,\nu)v$. On the other hand, using the vector decomposition \eqref{appendix.vector.decomposition}, we find
		\[\begin{split}
		\bangle{\nabla_{Y_1}N,\nu}&=\bangle{\nabla_{Y_1}(\cos\gamma X-\sin\gamma \overline{\nu}),\cos\gamma\overline{\nu}+\sin\gamma X}\\
		&=\cos^2\gamma\bangle{\nabla_{Y_1}X,\overline{\nu}}-\sin^2\gamma\bangle{\nabla_{Y_1}\overline{\nu},X}+\bangle{L,\nabla^{\partial\Sigma} \gamma}.\\
		&=\bangle{\nabla_{Y_1}X,\overline{\nu}}+\bangle{L,\nabla^{\partial\Sigma} \gamma}.
		\end{split}\]
		Here $L$ is a vector field along $\partial \Sigma$, and $|L|\le C=C(Y,X,\nu)$. Thus we conclude that
		
		\[\td{}{t}\bigg\vert_{t=0}\bangle{N_t,X_t}=-\sin\gamma\pa{v}{\nu}+(\cos\gamma)A(\nu,\nu)v+\secondfund(\overline{\nu},\overline{\nu})v+\bangle{L,\nabla^{\partial\Sigma}\gamma},\]
		as desired.
		
	\end{proof}
	
	The evolution equation of the mean curvature has been studied in many circumstances. We refer the readers to the thorough calculation in Proposition 16, \cite{Ambrozio15rigidity}:
	
	\begin{lemm}[Proposition 16 of \cite{Ambrozio15rigidity}]\label{lemma.appendix.3}
		Let $H_t$ be the mean curvature of $\Sigma_t$. Then
		\[\td{}{t}\bigg\vert_{t=0} H_t=\Delta_\Sigma v+(\Ric(N,N)+|A|^2)v-\bangle{\nabla_\Sigma H,Y_0}.\]
		In particular, if each $\Sigma_t$ has constant mean curvature, then
		\[\td{}{t} H_t=\Delta_{\Sigma_t} v_t+(\Ric(N_t,N_t)+|A_t|^2)v_t.\]
	\end{lemm}

	\bibliography{bib} 
	\bibliographystyle{amsalpha}
	
\end{document}